\documentclass[twoside,11pt]{amsart}
\usepackage{graphicx}
\usepackage{hyperref,color}
\usepackage{subfigure}

\usepackage{amsmath}
\usepackage{amsfonts}
\usepackage{amssymb}
\usepackage{latexsym}
\usepackage{upref}

\newtheorem{theorem}{Theorem}[section]

\newtheorem{corollary}[theorem]{Corollary}
\theoremstyle{definition}

\newtheorem{remark}[theorem]{Remark}
\allowdisplaybreaks
\numberwithin{equation}{section}

\def\SN{P_N}  
\def\AN{A_N}  
\def\SR{{\mathcal S}} 
\def\Id{I} 
\def\IN{\psi_N} 
\def\Iw{\IN[w]}
\def\Pw{\SN[w]}
\def\Tw{(I-\SN)[w]}  
\def\IminusIw{(I-\IN)[w]}
\def\Aw{\AN[w]}

\def\umk{u_{m_k}}
\def\ubar{\overline{u}}
\def\bubar{\overline{\mathbf u}}

\def\part{\partial}
\def\ddx{\frac{\part}{\part x}}
\def\ddt{\frac{\part}{\part t}}
\def\dt{\frac{d}{dt}}

\def\buildrul#1\under#2{\mathrel{\mathop{\null#2}\limits_{#1}}}
\def\hf{\frac{1}{2}}
\def\threehf{\frac{3}{2}}

\def\bx{{\mathbf x}}
\def\nablax{\nabla_{\bx}}

\def\x1{{x_1}}
\def\x2{{x_2}}

\def\pil{0}  
\def\pir{2\pi} 
\def\tpi{2\pi} 

\def\uN{u_N}
\def\um{u_m} 
\def\bu{{\mathbf u}}
\def\bv{{\mathbf v}}
\def\bw{{\mathbf w}}
\def\buN{{\mathbf u}_N}
\def\bum{{{\mathbf u}_m}}

\def\bx{\mathbf x}
\def\bk{\mathbf k}
\def\bj{\boldsymbol\ell} 
\def\bp{\mathbf p}
\def\bnu{\boldsymbol\nu} 

\def\iprod{\cdot}
\def\tor{{\mathbb T}^d}
\def\tpid{(\tpi)^d}
\def\dhf{\frac{d}{2}}

\def\aj{\ell}
\def\tthirds{\frac{2}{3}}
\newcommand*  {\Lr} {{\mathbb P}}
\def\del{\partial}

\newcommand{\R}{\Bbb{R}}

\newcommand{\myb}[1]{\textcolor{blue}{#1}}

\theoremstyle{definition}

\theoremstyle{remark}

\numberwithin{equation}{section}

\newcounter{asnr}

{\ifnum\value{asnr}=0 \stepcounter{asnr} 
  \begin{enumerate}[label=\textbf{A}.\arabic{enumi}]
    \else
    \begin{enumerate}[label=\textbf{A}.\arabic{enumi},resume] \fi}
{\end{enumerate}}

\begin{document}

\title[On the stability and instabilities of the Fourier method]{Stability  and spectral convergence of\\Fourier method for nonlinear problems:\\On the shortcomings of the $2/3$ de-aliasing method}

\author[C. Bardos]{Claude Bardos}\address[Claude Bardos]{
\newline University of Paris 7- Denis Diderot
\newline Laboratory Jacques Louis Lions, University of Paris 6
\newline Paris, France}
\email[]{claude.bardos@gmail.com}

\author[E. Tadmor]{Eitan Tadmor}\address[Eitan Tadmor]{
\newline Department of Mathematics
\newline Center of Scientific Computation and Mathematical Modeling (CSCAMM) 
\newline Institute for Physical sciences and Technology (IPST)
\newline University of Maryland
\newline MD 20742-4015, USA}
\email[]{tadmor@cscamm.umd.edu}

\date{\today}

\thanks{\textbf{Acknowledgment.}  E. T. Research was supported in part by NSF grants DMS10-08397, RNMS11-07444 (KI-Net) and   ONR grant \#N00014-1210318.}

\begin{abstract}
The high-order accuracy of Fourier method  makes it the method of choice in many large scale simulations. We discuss here the stability of Fourier method for nonlinear evolution problems, focusing on the two prototypical cases of the inviscid Burgers' equation and the multi-dimensional  incompressible Euler equations. The Fourier method for such problems with quadratic nonlinearities comes in two main flavors. One is the spectral Fourier method. The other is the $2/3$  pseudo-spectral Fourier method, where one removes the highest $1/3$ portion of the spectrum; this  is often the method of choice to maintain the balance of quadratic energy and avoid aliasing errors.\newline
Two main themes are discussed in this paper. First, we prove that as long as the underlying exact solution has a minimal $C^{1+\alpha}$ spatial regularity, then both  the spectral and the $2/3$ pseudo-spectral Fourier methods are  stable. Consequently, we  prove their spectral convergence for smooth solutions of the inviscid Burgers equation and the  incompressible Euler equations. On the other hand, we prove that after a critical time at which  the underlying solution lacks sufficient smoothness, then both  the spectral and the $2/3$ pseudo-spectral Fourier methods exhibit nonlinear instabilities which are realized through spurious oscillations. 
In particular, after shock formation in inviscid Burgers' equation, the total variation of  bounded (pseudo-) spectral Fourier solutions \emph{must} increase with the number of increasing modes and we stipulate the analogous situation occurs with the 3D incompressible Euler equations: the limiting Fourier solution is shown to enforce $L^2$-energy conservation, and the  contrast with energy dissipating Onsager solutions  is reflected through spurious oscillations.
\end{abstract}

\maketitle
\tableofcontents
\section{Introduction}\label{sec:intro}

Spectral methods are often the methods of choice when high-resolution solvers are sought for nonlinear time-dependent problems.
Here, we are concerned  with the stability and convergence of Fourier method for PDEs with quadratic nonlinearities: we focus our attention on the prototypical Cauchy problems for the inviscid Burgers' equation and the incompressible Euler equations. 

The Fourier methods for problems involving quadratic nonlinearities come in two main flavors: the \emph{spectral Fourier method} and the 2/3 smoothing of \emph{pseudo-spectral Fourier method}. The spectral Fourier method is realized in terms of $N$-degree Fourier expansions,
$\buN(\bx,t)=\mathop{\sum}_{|\bk|\leq N} \widehat{\bu}_{\bk}(t)e^{i\bk\iprod\bx}$, where $\widehat{\bu}_{\bk}(t)$  are the Fourier moments of $\bu(\bx,t)$
\[
\widehat{\bu}_{\bk}(t) = \frac{1}{\tpid}\int_{\tor}\!\!\bu(\bx)e^{-i\bk\iprod\bx}d\bx, \quad \bk :=(k_1,\ldots,k_d) \in {\mathbb Z}^d.
\]
The computation of these moments in nonlinear problems is carried out by convolutions. These can be avoided when the $\widehat{\bu}_{\bk}$'s are replaced by the  discrete Fourier coefficients, sampled at the $(2N+1)^d$ equally spaced grid points 
\[
\widetilde{\bu}_{\bk}(t) = \left(\frac{1}{2N+1}\right)^d \sum_{\bx_{\bnu}\in \tor_\#} \bu(\bx_{\bnu},t)e^{-i\bk\iprod\bx_{\bnu}},
\qquad \bx_{\bnu}=\frac{\tpi\bnu}{2N+1},
\]
where $\tor_\#$ is the discrete torus, 
\[
\tor_\#:=\left\{\bx_{\bnu}\ \large| \ \bx_{\bnu}=\frac{\tpi\bnu}{2N+1}, \quad \bnu=(\nu_1,\ldots,\nu_d), \ 0\leq \nu_j\leq 2N\right\}.
\]
The pseudo-spectral Fourier method is realized in terms of the corresponding expansion, $\buN(\bx,t)=\sum_{|\bk|\leq N} \widetilde{\bu}_{\bk}(t)e^{i\bk\iprod\bx}$. Here, we have the advantage that nonlinearities are computed as exact pointwise quantities at the grid points $\{\bx_{\bnu}\}_{\bnu}$, but new aliasing errors are introduced. To avoid aliasing errors and their potential instabilities, high mode smoothing is implemented, which results in the so-called $2/3$-smoothing of pseudo-spectral Fourier method: it is realized in terms of the $2N/3$-degree expansion, $\buN(\bx,t)=\sum_{|\bk|\leq 2N/3} \sigma_{\bk}\widetilde{\bu}_{\bk}(t)e^{i\bk\iprod\bx}$. This is the spectral method of choice in many  time-dependent problems with quadratic nonlinearities.

To put our discussion into perspective we begin, in section \ref{sec:lin}, by recalling the linear setup of standard transport equation. The spectral Fourier method is $L^2$-stable. But the pseudo-spectral Fourier method is not \cite{GHT94}: it is only \emph{weakly} stable, due to amplification of aliasing errors when the underlying solution lacks sufficient smoothness. Strong $L^2$-stability is regained with the  $2/3$-smoothing of pseudo-spectral Fourier method, \cite{Tad87}; in the linear setup,  the de-aliasing in the $2/3$-method introduces sufficient  smoothness to maintain convergence. This is one of the main two themes of our results on nonlinear problems: sufficient smoothness guarantees stability and hence spectral convergence. In section \ref{sec:burgers} we explore this issue in the context of inviscid Burgers equations, proving  that as long as the solution of the inviscid Burgers equation remains smooth, $u(\cdot,t)\in C^{1+\alpha}_x$, then both the spectral  and the $2/3$-pseudo-spectral Fourier approximations, $\uN(\cdot,t)$, converge to the exact solution. Moreover, they enjoy \emph{spectral convergence rate}, namely, the convergence rate grows with  the increasing smoothness of $u(\cdot,t)$,
\begin{eqnarray*}
\lefteqn{\int |\uN(x,t)-u(x,t)|^2dx} \\
& &  \lesssim e^{\displaystyle \int_0^t\!\!\|u_x(\cdot,\tau)\|_{L^\infty}d\tau} \!\!\cdot \left(N^{-2s}\|u(\cdot,0)\|^2_{H^s} + N^{\threehf-s}\max_{\tau\leq t}\|u(\cdot,\tau)\|_{H^s}\right)\!, \quad s>\threehf.
\end{eqnarray*}

A similar statement of spectral convergence holds for the spectral and $2/3$ pseudo-spectral Fourier approximations $\buN$ of  the incompressible Euler equations: in section \ref{sec:smoothEuler} we prove that as long as $\bu(\cdot,t)$ remains sufficiently smooth solution of the $d$-dimensional Euler equations, $\bu(\cdot,t)\in C^{1+\alpha}_x$, then  
\begin{eqnarray*}
\lefteqn{\|\buN(\cdot,t)-\bu(\cdot,t)\|^2_{L^2}} \\
& &  \lesssim e^{\displaystyle 2\!\int_0^t\!\!\|\nablax\bu(\cdot,\tau)\|_{L^\infty}d\tau} \!\!\cdot \left(N^{-2s}\|\bu(\cdot,0)\|^2_{H^s} + N^{\frac{d}{2}+1-s}\max_{\tau\leq t}\|\bu(\cdot,\tau)\|_{H^s}\right)\!, \ s>\frac{d}{2}+1.
\end{eqnarray*} 
These results support the superiority of spectral methods for problems with smooth solutions. When dealing with  solutions  which lack smoothness, however, both the spectral and $2/3$ pseudo-spectral Fourier methods suffer nonlinear instabilities. This is the other main theme of the paper, explored  in the context of the inviscid Burgers equation and the incompressible Euler equations in the respective  sections \ref{sec:2-3rd} and   \ref{sec:nonsmoothEuler}. In particular,   we prove that  after shock formation, the spectral and $2/3$ pseudo-spectral bounded approximations of the inviscid Burgers solution  \emph{must} produce spurious oscillations as their total variation must increase, $\|\uN(\cdot,t)\|_{TV} \gtrsim \sqrt[4]{N}$. This is deduced by contradiction:  in theorem \ref{thm:instab}  below we prove, using compensated compactness arguments, that an $L^2$-weak limit of slowly growing TV Fourier solutions, $\ubar=\text{w}\lim \uN$,  must be an $L^2$-energy \emph{conservative} solution, which cannot hold once shocks are formed.\newline
A similar scenario arises with the  Euler solutions where the spectral and the ($2/3$ pseudo-)spectral approximations of Euler equations enforce conservation of the $L^2$-energy.  Although there is  no known energy dissipation-based selection principle to identify a unique solution of  Euler equations within the class of ``rough" data (similar to the entropy dissipation selection principle for Burgers' equations), nevertheless we argue that the $L^2$-energy conservation of the (pseudo-)spectral approximations may be responsible to their  unstable behavior. 
While $L^2$-energy conservation holds for weak solutions with a minimal degree of $1/3$-order of smoothness (Onsager's  conjecture proved  in \cite{Ey94,CET94,BT10}), there are  experimental and
numerical evidence for the other part of Onsager's conjecture that anomalous dissipation of energy shows  up for ``physical-turbulent"  $L^2$-solutions of Euler equations \cite{Co07}. Whether this observed anomalous dissipation of energy should be due to spontaneous appearance of singularities in smooth solutions of the Euler equation or to the fact that physical initial data may be rough is a completely open problem. However  after several preliminary breakthrough \cite{Sc93} and \cite{Sh97} the following fact are now well established. Indeed, there are infinitely many initial data (which of course are not regular) leading to infinitely many weak  Euler solutions with energy loss \cite{DeLS12}. In particular there are  solutions which for almost  every time  belong to the critical regularity $C^{\frac13-\epsilon}$ \cite{Buck13}. Thus, if the numerical method captures such ``rough"  solutions then the ``unphysical" conservation of energy which is enforced at the spectral level has to vanish at the limit, leading to spurious oscillations.

We close this paper with two complementary results. 
 The nonlinear instability results in sections \ref{sec:2-3rd} and \ref{sec:nonsmoothEuler} emphasize  the  competition between spectral convergence for smooth solutions vs. nonlinear instabilities for problems which lack sufficient smoothness. 
In section \ref{sec:sv} we discuss the class of \emph{spectral viscosity} (SV) methods which entertain both --- spectral convergence and nonlinear stability, \cite{Tad89,Tad93b, KK00,GP03,SS07,AX09}. This is  achieved by adding a judicious amount of spectral viscosity at the high-portion of the spectrum without sacrificing the spectral accuracy at the lower portion of the spectrum. Finally, we note that the above stability and instability results are not necessarily restricted to quadratic nonlinearities: in section \ref{sec:isen} we prove the stability of Fourier method for smooth solutions of the nonlinear isentropic equations.

\subsection{Spectral convergence}\label{sec:spec-convergence}
Expressed in terms of the Fourier coefficients, $\widehat{w}(\bk)$, the \emph{spectral} Fourier  projection $\Pw(\bx)$ of $w\in L^1[\tor]$ is given by
\[
\Pw(\bx)=\sum_{|\bk|\leq N}\widehat{w}(\bk)e^{i\bk\iprod\bx}, \qquad \widehat{w}(\bk):=\frac{1}{\tpid}\int_{\tor}\!\!w(\bx)e^{-i\bk\iprod\bx}d\bx, \quad \bk :=(k_1,\ldots,k_d) \in {\mathbb Z}^d.
\]
The convergence rate of the \emph{truncation error},
\begin{equation}\label{eq:tr}
\Tw(\bx):=\sum_{|\bk|\geq N}\widehat{w}(\bk)e^{i\bk\iprod\bx},
\end{equation}
is as rapid as the global smoothness of $w$ permits (and observe that the degree of smoothness is allowed to be negative),
\[
\|\Tw\|_{\dot{H}^r} \leq N^{r-s}\|w\|_{\dot{H}^s}, \qquad s>r \in {\mathbb R}; 
\]
in particular, 
\begin{equation}\label{eq:maxHs}
\max_\bx|(I-\SN)[w](\bx)| \lesssim N^{\dhf-s}\|w\|_{H^s}, \qquad s>\dhf.
\end{equation}
These are statements of \emph{spectral convergence} rate: the smoother $w$ is, the faster is the convergence rate of $\Tw \rightarrow 0$. In practice, one recovers exponential convergence which characterizes analytic regularity or at least root-exponential rate for  typical compactly supported Gevrey-regular data, \cite{Tad07}.

\subsection{Aliasing}\label{sec:aliasing}
Set $h:=\frac{\tpi}{2N+1}$ as a discrete spacing. If we replace the integrals with  quadrature  based on sampling  $w$ at the $(2N+!)^d$equi-spaced points, $\displaystyle \bx_{\bnu}:=\bnu h, \ \bnu:=(\nu_1,\ldots,\nu_d)\in \{0,2N\}^d$, we obtain the  \emph{pseudo-spectral}  Fourier projection,  
\[
\Iw(\bx)=\sum_{|\bk|\leq N}\widetilde{w}(\bk)e^{i\bk\iprod\bx}, \qquad \widetilde{w}(\bk):=\left(\frac{h}{\tpi}\right)^{\!\!d}\sum_{\bx_{\bnu} \in  \tor_\#}w(\bx_{\bnu})e^{-i\bk\iprod\bx_{\bnu}}, \quad |\bk| \leq N.
\]
Here, $\widetilde{w}(\bk)$, are the  discrete Fourier coefficients\footnote{There is a slight difference between the formulae based on an even and an odd number of points;  we chose to continue with the slightly simpler notations associated with an odd number of points.}.
The mapping $w\mapsto \Iw$ is a projection:  $\Iw(\bx)$ is  the unique $N$-degree \emph{trigonometric interpolant} of $w$ at the $(2N+1)$-gridpoints, $\Iw(\bx_{\bnu})=w(\bx_{\bnu}), \ |\bnu|\leq  2N$.
The dual statement of the last equalities is the Poisson summation formula, which determines the discrete $\widetilde{w}(\bk)$'s in terms of the exact Fourier coefficients, $\widehat{w}(\bk)$'s,
\[
\widetilde{w}(\bk)=\widehat{w}(\bk) + \sum_{\bj\neq 0} \widehat{w}(\bk+\bj(2N+1)), \qquad |\bk|\leq N,
\]
where summation runs over all $d$-tuples, $\bj=(\ell_1,\ldots,\ell_d)\neq 0$.
It shows that all the Fourier coefficients with wavenumber $\bk[mod (2N+1)]$ are ``aliased" into the same discrete Fourier coefficient, $\widetilde{w}_\bk$. It follows that the  interpolation error consists of two main contributions, 
\[
\IminusIw \equiv \Tw+\Aw,
\]
 where in addition to the truncation error $\Tw$ in (\ref{eq:tr}),  we now have the  \emph{aliasing error},
\begin{equation}\label{eq:al}
\AN[w](\bx):=\sum_{|\bk|\leq N}\Big(\sum_{|\bj|\geq 1}\widehat{w}(\bk+\bj(2N+1))\Big)e^{i\bk\iprod\bx}.
\end{equation}
Both, $\Tw$ and $\Aw$, involve high modes, $\widehat{w}(\bp),\ |\bp|\geq N$. Consequently, if the function $w(\cdot)$ is sufficiently smooth then they have exactly the same spectrally small size, e.g. \cite[\S2.2]{Tad94}
\[
\|\AN[w]\|_{H^s} \lesssim \|\Tw\|_{H^s} \lesssim N^{s-r}\|w\|_{H^r}, \qquad r>s>\dhf.
\]
Since  the truncation error is orthogonal to the computational $N$-space whereas the aliasing error is not the situation is different if $w$ lacks smoothness. Then  aliasing  and truncation errors may have a completely different influence  on the question of computational stability. One such case is encountered with the  stability question of spectral vs. pseudo-spectral approximations of hyperbolic equations.

\section{Fourier method for linear equations --- weak instability for $L^2$-data}\label{sec:lin}
\subsection{The spectral approximation: stability and convergence.}
We want to solve the   $\tpi$-periodic scalar hyperbolic equation
\begin{equation}\label{eq:pde}
\frac{\partial}{\partial t}u(x,t) + \frac{\partial}{\partial x}\big(q(x)u(x,t)\big)=0, \qquad x\in {\mathbb T}([0,2\pi)), \ \ q\in C^1[\pil,\pir],
\end{equation}
subject to prescribed initial conditions, $u(\cdot,0)$, by  the spectral Fourier method. To this end we approximate the spectral projection of the exact solution, $\SN u(\cdot,t)$, using an $N$-degree polynomial, $\uN(x,t)=\sum_{|k|\leq N}\widehat{u}_k(t)e^{ikx}$, which is governed by the semi-discrete approximation, \cite{Or72,KO72,GO77}
\begin{equation}\label{eq:splina}
\frac{\partial}{\partial t}\uN(x,t)+\frac{\partial}{\partial x}\SN[q(x)\uN(x,t)]=0.
\end{equation}
The approximation is realized as a convolution in Fourier space
\[
\frac{d}{dt}\widehat{u}_k(t) = ik\sum_{|j|\leq N}\widehat{q}(k-j)\widehat{u}_j(t), \qquad k=-N,\ldots,N,
\]
which amounts to a system of $(2N+1)$ ODEs for the computed $\widehat{\bf u}(t):= (\widehat{u}_{-N}(t),\ldots,\widehat{u}_N(t))^\top$.

The $L^2$-stability of \eqref{eq:splina} is straightforward: though the truncation error which enters (\ref{eq:splina}), $\partial_x(I-\SN)[q(x)\uN(x)]$ need not be small, it  is orthogonal to the $N$-space, and hence,
\begin{eqnarray}
\frac{1}{2}\frac{d}{dt}\|\uN(\cdot,t)\|^2_{L^2} &=& -\int_{\pil}^{\pir} \uN\frac{\partial}{\partial x}\SN[q(x)\uN]dx  
 =  \int\frac{\partial \uN}{\partial x} \SN[q(x)\uN]dx  \nonumber  \\
& =&  
\int\frac{\partial \uN}{\partial x} q(x)\uN dx = -\frac{1}{2}
\int q'(x)\uN^2dx  \label{eq:spstab} \\
&\leq & 
\frac{1}{2}\max_x|q'(x)|\times\|\uN(\cdot,t)\|^2_{L^2}. \nonumber
\end{eqnarray}
This  yields the $L^2$-stability bound,
\begin{equation}\label{eq:spl2stab}
\|\uN(\cdot,t)\|^2_{L^2} \leq e^{q'_\infty t}\|\uN(\cdot,0)\|^2_{L^2}, \qquad q'_\infty:=\max_x|q'(x)|.
\end{equation}

To convert this stability bound into a spectral convergence rate estimate, consider the difference between the spectral method (\ref{eq:splina}) and the $\SN$ projection of the underlying equation (\ref{eq:pde}): one finds that  $e_N:=\uN-\SN u$, satisfies the error equation
\[
\frac{\partial}{\partial t}e_N(x,t)+\frac{\partial}{\partial x}\SN[q(x)e_N(x,t)]= -\ddx \SN\big[q(x)(I-\SN)[u](x,t)\big].
\]
The $L^2$-stability  (\ref{eq:spl2stab}) implies the  error estimate,
\[
\int|\uN(x,t)-\SN u(x,t)|^2dx \lesssim e^{q'_\infty t}\left(\|(I-\SN)u(\cdot,0)\|^2_{L^2}  + N^2\max_{\tau\leq t}\|(I-\SN)[u](\cdot,\tau)\|^2_{L^2}\right).
\] 
This quantifies  the \emph{spectral convergence} of the Fourier method (\ref{eq:splina}): the convergence rate increases together with the increasing order of smoothness of the solution,
\begin{equation}\label{eq:rate}
\|\uN(\cdot,t)-u(\cdot,t)\|_{L^2} \lesssim e^{\hf q'_\infty t}\left(N^{-s}\|u(\cdot,0)\|_{H^s}  + N^{1-s}\max_{\tau\leq t}\|u(\cdot,\tau)\|_{H^s}\right), \quad s>1.
\end{equation}
In practice, one recovers   exponential convergence for analytic solutions (and root-exponential convergence for more general Gevrey data).
\subsection{The pseudo-spectral approximation: aliasing and weak stability.}
We now consider pseudo-spectral Fourier approximation of (\ref{eq:pde}). As before, we use an $N$-degree polynomial, $\uN(x,t)=\sum_{|k|\leq N}\widehat{u}_k(t)e^{ikx}$, as an approximation for $\IN u(\cdot,t)$, which is governed by the semi-discrete approximation, \cite{KO72,GO77},
\begin{equation}\label{eq:psplina}
\frac{\partial}{\partial t}\uN(x,t)+\frac{\partial}{\partial x}\IN[q(x)\uN(x,t)]=0.
\end{equation}
This equation can be realized in physical space 
\[
\dt \uN(x_j,t)=\!\!\!\!\sum_{k=-N}^N\!\!\!\!ik \widetilde{(q\uN)}_k e^{ikx_j}, \quad \widetilde{(q\uN)}_k\!\!=\!\!\frac{h}{\tpi}\sum_{\nu=0}^{2N}q(x_\nu)\uN(x_\nu)e^{-ikx_\nu}.
\]
It amounts to a  system of $(2N+1)$ ODEs for the computed gridvalues 
${\bf u}(t):=\big({u}(x_0,t),\ldots, {u}(x_{2N},t)\big)^\top$ 
\[
\frac{d}{dt}{\bf u}(t) = DQ{\bf u}(t), \qquad D_{jk}=\frac{(-1)^{j-k}}{2\sin\left(\frac{x_j-x_k}{2}\right)}, \ \ Q = 
\left(\begin{array}{cccc}q(x_0)& 0&\ldots&0 \\
0 &\ddots & 0 & \vdots\\ 
\vdots & & \ddots & 0\\
0 & 0 & \ldots & q(x_{2N})\end{array}\right)
\]
Here $D$ is the Fourier differentiation matrix and $Q$ signifies pointwise multiplication with $q(x)$. 

To examine the stability of (\ref{eq:psplina}) we repeat the usual $L^2$-energy argument for the  spectral approximation in (\ref{eq:splina}): decompose $\IN=\SN+\AN$, to find 
\begin{equation}\label{eq:pspener}
\frac{1}{2}\dt \|\uN(\cdot,t)\|^2_{L^2}  = \overbrace{\int \uN\frac{\partial}{\partial x} {\SN}[q(x)\uN]dx}^{{\rm a \ bounded \ term} - (\ref{eq:spstab})} + 
\overbrace{\int \uN\frac{\partial}{\partial x} \AN[q(x)\uN]dx}^{{\rm contribution  \ of \ aliasing}}
\end{equation}
The first term on the right consists of truncation error which, by (\ref{eq:spstab}), does not exceed $\lesssim \|\uN(\cdot,t)\|^2$. Thus, the stability of the Fourier approximation (\ref{eq:psplina}) depends solely on the aliasing contributions, $\AN[q(x)\uN]$: using (\ref{eq:al}) to expand the second term on the right, we find 
\begin{equation}\label{eq:semi_aliasing}
\int \uN\frac{\partial}{\partial x}{A_N}[q(x)\uN]dx = 2\pi i\!\!\!\!\sum_{|j|,|k|\leq N}
\overline{\widehat{u}}_j(t) \widehat{u}_{k}(t) \,(j-k)\cdot\!\!\sum_{{\aj\neq 0}}\widehat{q}\left({j-k+\aj(2N+1)}\right). 
\end{equation}
Observe that the terms on the right, $\sum_{{\aj\neq 0}}\widehat{q}\left({j-k+\aj(2N+1)}\right)$, are of order ${\mathcal O}(N)$ for $|j-k|\sim 2N, \aj=\pm 1$, and this can occur \emph{only} for high wavenumbers, $|j|\sim |k| \sim N$. Thus, there is possible ${\mathcal O}(N)$ amplification of the high Fourier modes, $|\widehat{u}_j(t)|, \ |j|\sim N$.  Unfortunately, these Fourier modes  need not be small due to lack of apriori smoothness, and aliasing may render the Fourier method as unstable.

Indeed, when $q(x)$ changes sign, the exact solution of (\ref{eq:pde})  develops large gradients and consequently, the Fourier method does experience 
spurious oscillations precisely because of aliasing errors which are ignited due to lack of smoothness. 
To demonstrate the exact mechanism of this type of instability\footnote{To demonstrate the instability of the Fourier method (\ref{eq:psplina}), one needs to consider $q(x)$ which changes sign. Otherwise, if $q(x)$ is, say, positive, then $DQ$ is similar to the anti-symmetric matrix $
{\mathcal A}:=\sqrt{Q}D\sqrt{Q}$ and stability follows  since $exp({\mathcal A}t)$ is unitary 
in ${\mathbb C}^{2N+1}$ for the scalar product $(QU_N,V_N)$ \cite{KO72}.}, we consider the example $q(x)=\sin(x)$,
\begin{equation}\label{eq:pspsin}
\frac{\partial}{\partial t}\uN(x,t)+\frac{\partial}{\partial x}\IN[\sin(x)\uN(x,t)]=0
\end{equation}
The analysis follows \cite{GHT94}. Fourier transform of (\ref{eq:pspsin}) yields
\[
\frac{d}{dt} \widehat{u}_k (t) = {k\over 2} [\widehat{u}_{k-1} (t) - \widehat{u}_{k+1}(t)], \quad k=-N,\dots, N,
\]
and its imaginary part, $b_k(t):=\Im\widetilde{u}_k(t)$, reads
\begin{equation}\label{eq:pspb}
\frac{d}{dt} b_k(t)  =  {k\over 2} \left[b_{k-1} (t) - b_{k+1}(t)\right], \ \  
b_{N+1} +b_N=0.  
\end{equation} 

The last set of ODEs is at the heart of matter.
A straightforward energy estimate yields the \emph{lower-bound}
$\displaystyle \frac{d}{dt} \|{\bf b}(t)\|^2 \geq - \|{\bf b}(t)\|^2+  N b^2_N(t)$ for ${\bf b}(t):=(b_1(t),\ldots,b_N(t))^\top$.
Does $b_N(t)$ grow? on the one hand, numerical simulations in Figure \ref{fig:a} show that it does.
On the other hand, \emph{if} the solution is smooth enough, then aliasing errors are negligible: in Figure \ref{fig:b} for example, the solution subject to  initial data $|\widehat{u}_k(0)| \sim |k|^{-3}$ remain smooth and the spurious mode, $|b_N(t)| \sim N^{-2}$, decay sufficiently fast so that ${\bf b}(t)$ remains bounded.

\begin{figure}[ht]
\begin{center}
\begin{tabular}{ccc}
\includegraphics[scale=0.3]{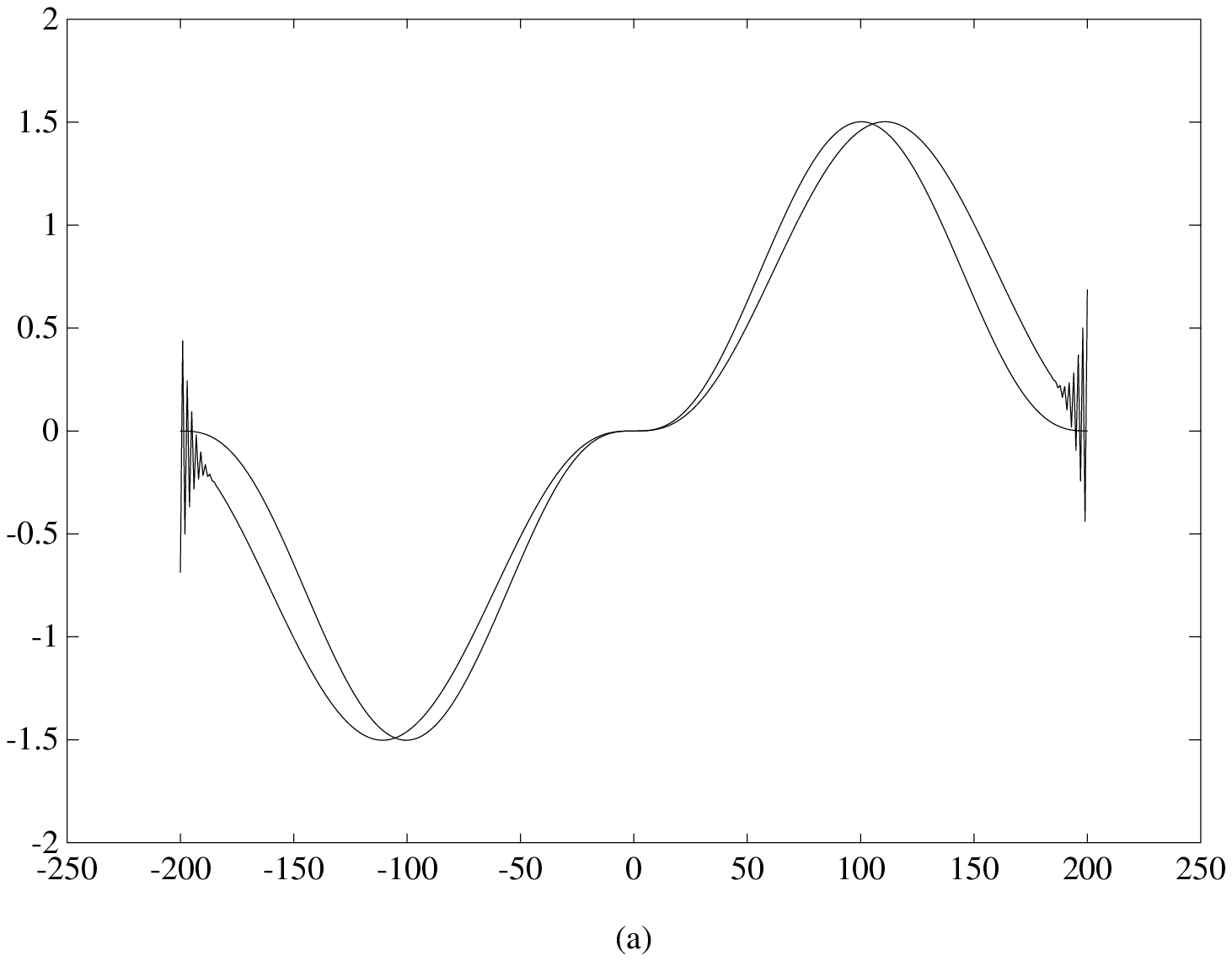} &
\includegraphics[scale=0.3]{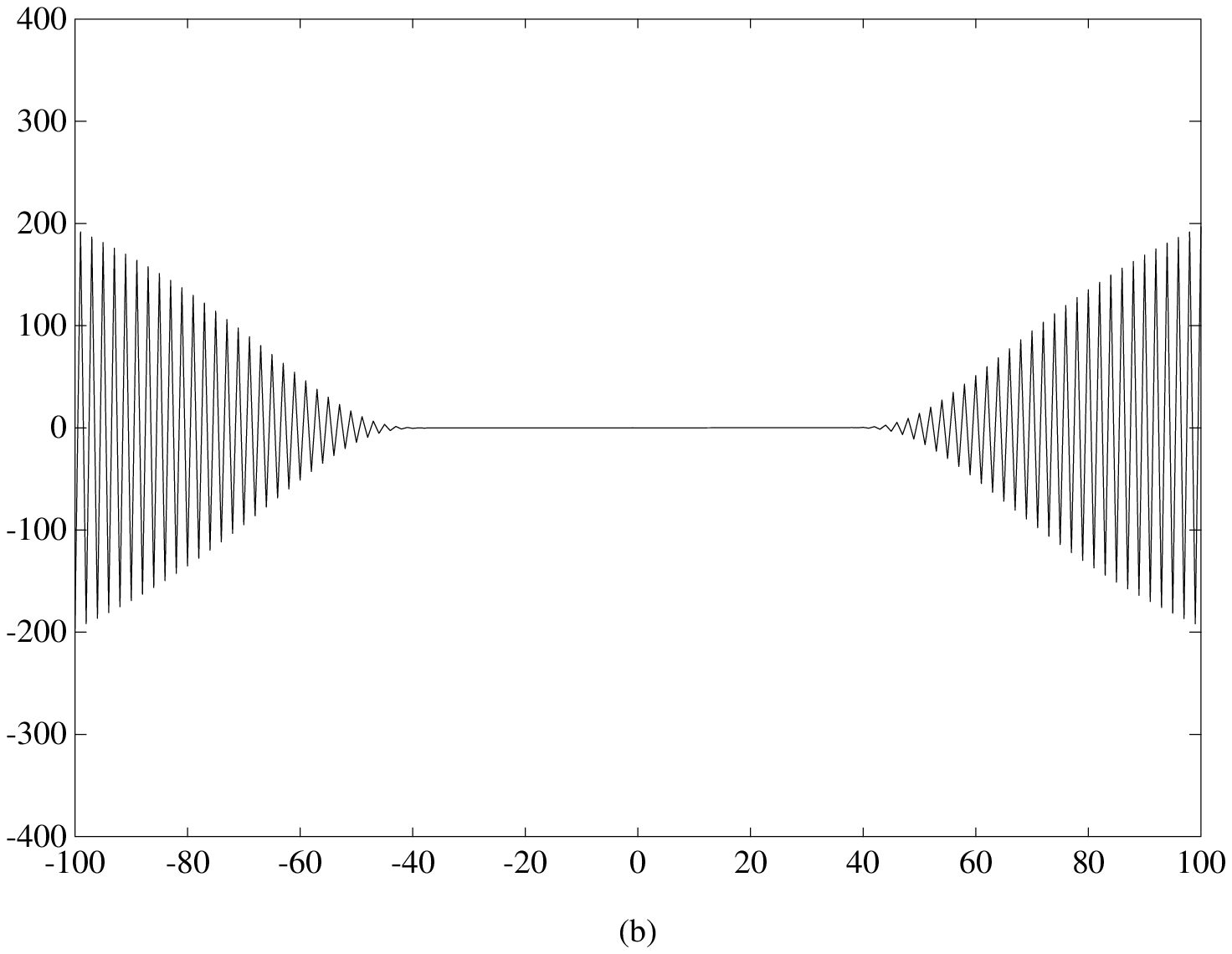} &
\includegraphics[scale=0.3]{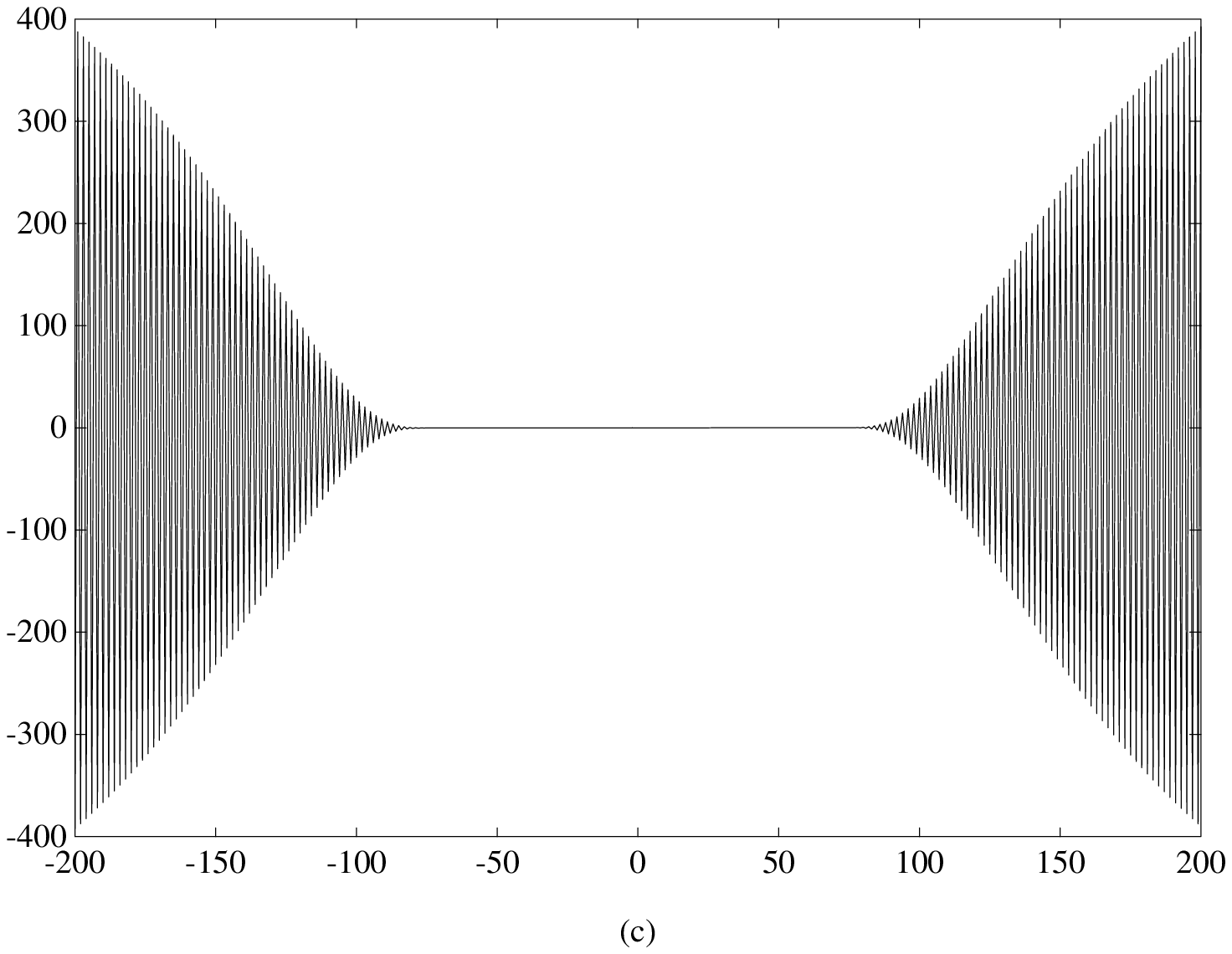}
\end{tabular}
\end{center} 
\caption{ $b_k(t)=\Im\widehat{u}_{k}(t)$: Fourier method subject to $\hat{u}_{k}(0)=x_{k}^{3}(\pi-x_{k})^{3}/20$.\newline (a) $(t,N)=(0.1, 200)$; \  
(b) $(t,N)=(1., 100)$; \ (c) $(t,N)=(1., 200)$.}\label{fig:a}
\end{figure}

\begin{figure}[ht]
\begin{center}
\begin{tabular}{ccc}
\includegraphics[scale=0.3]{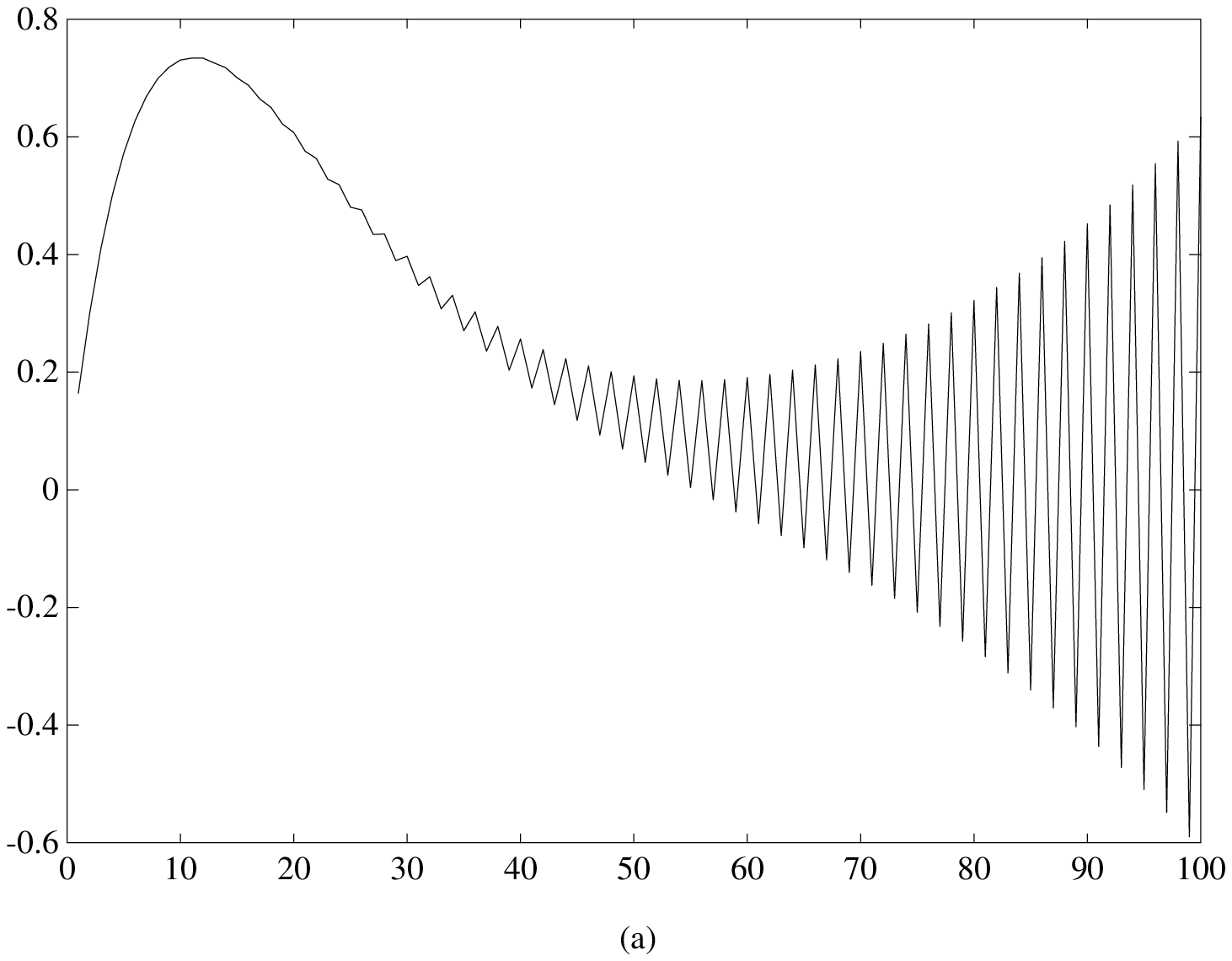} &
\includegraphics[scale=0.3]{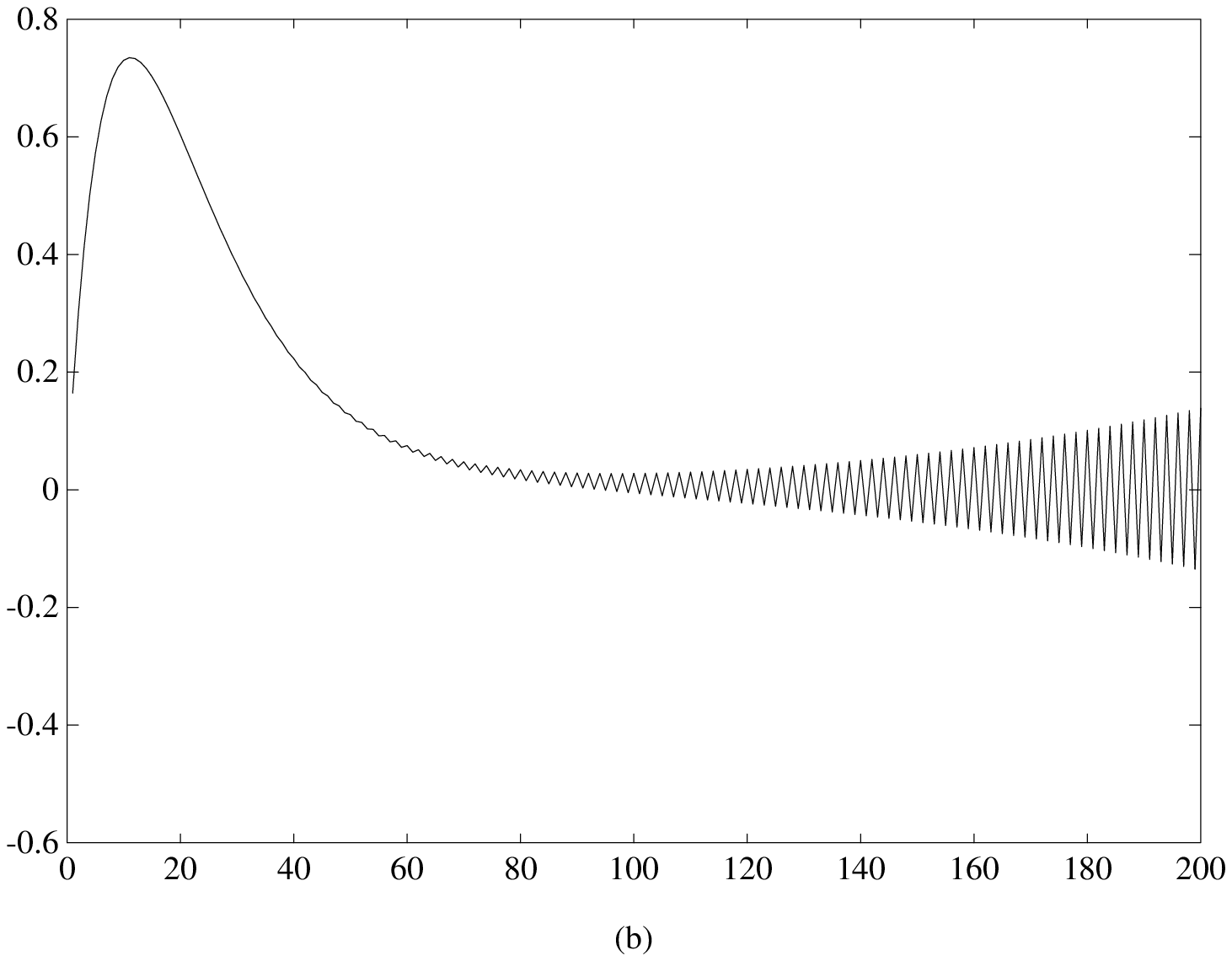} &
\includegraphics[scale=0.3]{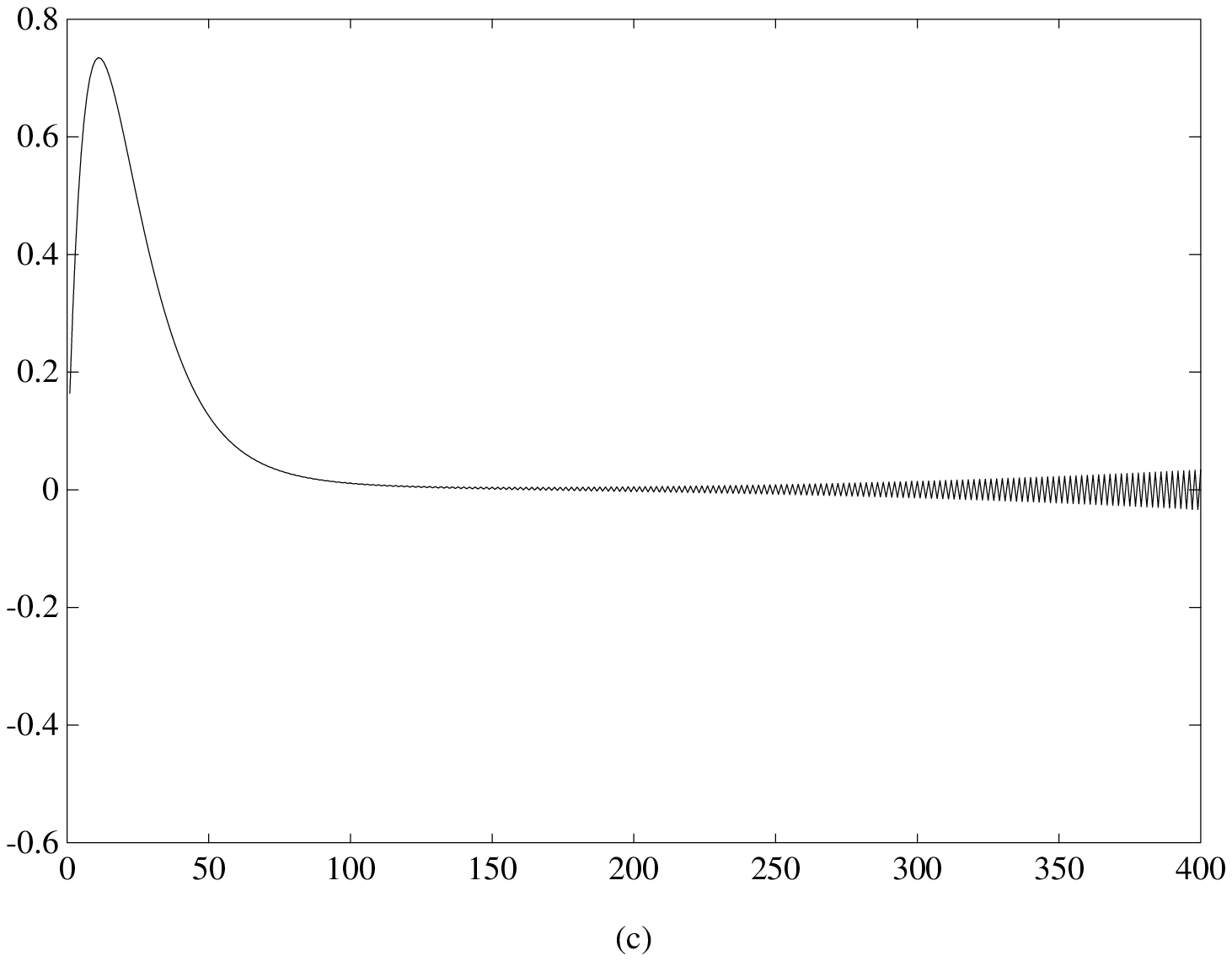} 
\end{tabular}
\end{center}
\caption{$\Im\hat{u}_{k}(t)$ computed at $t=3$: Fourier solution of $u_t=(\sin(x)u)_{x}, \quad
\widehat{u}_{k}(0)\sim{k^{-3}}$. 
 (a) with $N=100$; \
(b) with $N=200$; \
(c) with $N=800$: $b_N(t) \sim \pm  N^{-2}$.}\label{fig:b}  
\end{figure}

In practice, however, even if the solution of (\ref{eq:pde}) remains smooth,  it develops large gradients of order $|\widehat{\bf u}(t)| \sim exp(q'_\infty t)$ when $q(x)$ changes sign. These large gradients require $N\gg e^t$ modes in order to resolve these gradients; otherwise, the exact solution $u(\cdot,t)$ remains  under-resolved by the pseudo-spectral Fourier approximation. Observe that the Fourier method requires an increasingly large number of modes before it can resolve the  underlying solution, $u(\cdot,t)$. Without it, the  under-resolved Fourier approximation contains ${\mathcal O}(1)$ high modes, $|\widehat{u}_N(t)| \sim |b_N(t)|$, amplified by a factor of order ${\mathcal O}(N)$, yielding the spurious oscillations noticeable in Figure \ref{fig:c}.
Thus, aliasing errors cause the Fourier solution to grow due to  \emph{lack of resolution}. The precise growth is ``encoded'' in the Fourier equations whose imaginary part is governed by (\ref{eq:pspb}).  To this end,  set
 $v_{k}(t):=(-1)^{k}b_{k}(t)$. Then (\ref{eq:pspb}) is converted into
\[
\frac{d}{dt}  v_{k}(t) = 
x_{k} \frac{v_{k+1}(t)-v_{k-1}(t)}{2\Delta x}, 
 \quad   v_{N+1}(t)=v_{N}(t), \ \ x_{k}:=k\Delta x \in [0,1], \ \Delta x:=h/\pi.
\]
This can be viewed as an approximation to the linear equation
$  \del_t v(x,t) = x \del_x v(x,t), \ 0 \leq x \leq 1$ augmented with the boundary condition $\del_x v(1,t)=0$. This is an \emph{ill-posed} problem due to the extrapolation at the inflow boundary $x=1$,  and consequently, its numerical approximation grows $\|b_N(t)\|=\|v_N(t)\| \sim \sqrt{N}$, e.g., \cite{Tad83}.
The detailed analysis carried out in \cite{GHT94} shows that there is a \emph{weak} instability, where $\sim 1-e^{-t}$ fraction of the highest modes experience amplification of order ${\mathcal O}(N)$, which ends with the stability estimate

\[
\|\uN(\cdot,t)\|_{L^2} \lesssim e^{C q'_\infty t} N \|\uN(\cdot,0)\|_{L^2}.
\]
The corresponding  error estimate for the pseudo-spectral approximation reads \cite[theorem 4.1]{GHT94}
\[
\|\uN(\cdot,t)-\IN u(\cdot,t)\|_{L^2} \lesssim e^{{C_s}q'_\infty t}\left(N^{1-s}\|u(\cdot,0)\|_{H^s}  + N^{2-s}\max_{\tau\leq t}\|u(\cdot,\tau)\|_{H^s}\right), \qquad s>2,
\]
reflecting the  loss of power on $N$ when compared with the spectral estimate (\ref{eq:rate}).

\begin{figure}[ht]
\begin{center}
\begin{tabular}{ccc}
\includegraphics[scale=0.3]{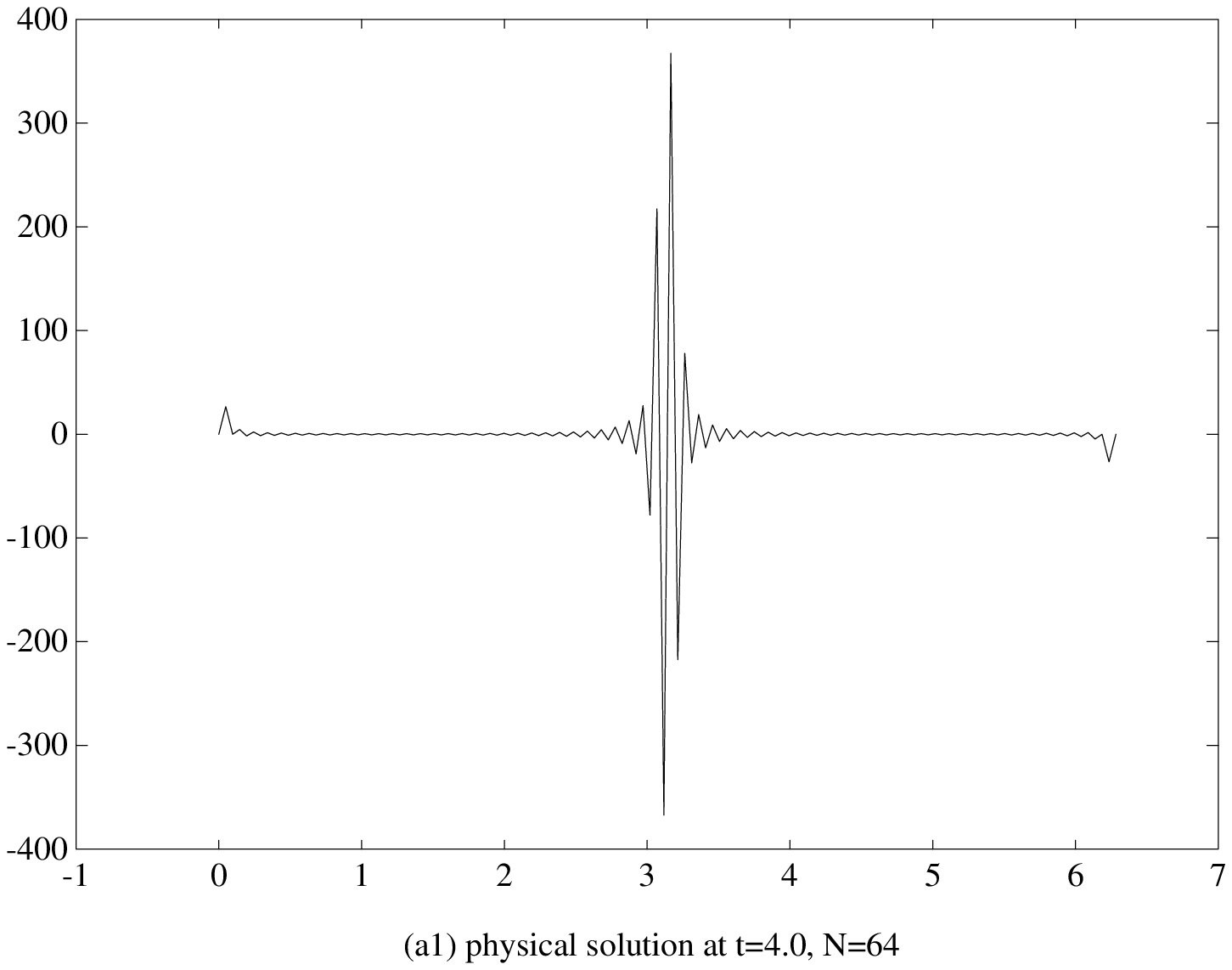}& \includegraphics[scale=0.3]{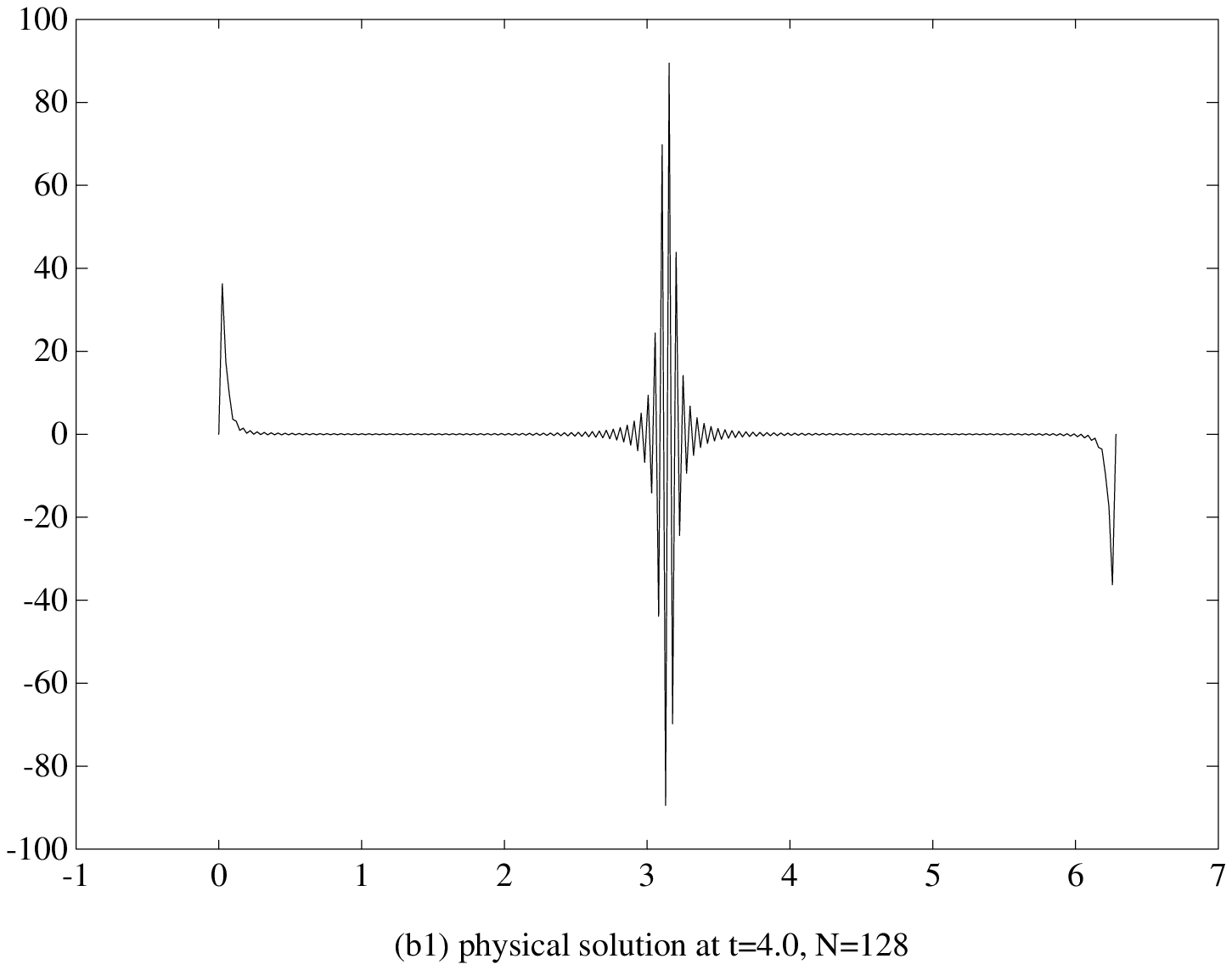}& \includegraphics[scale=0.3]{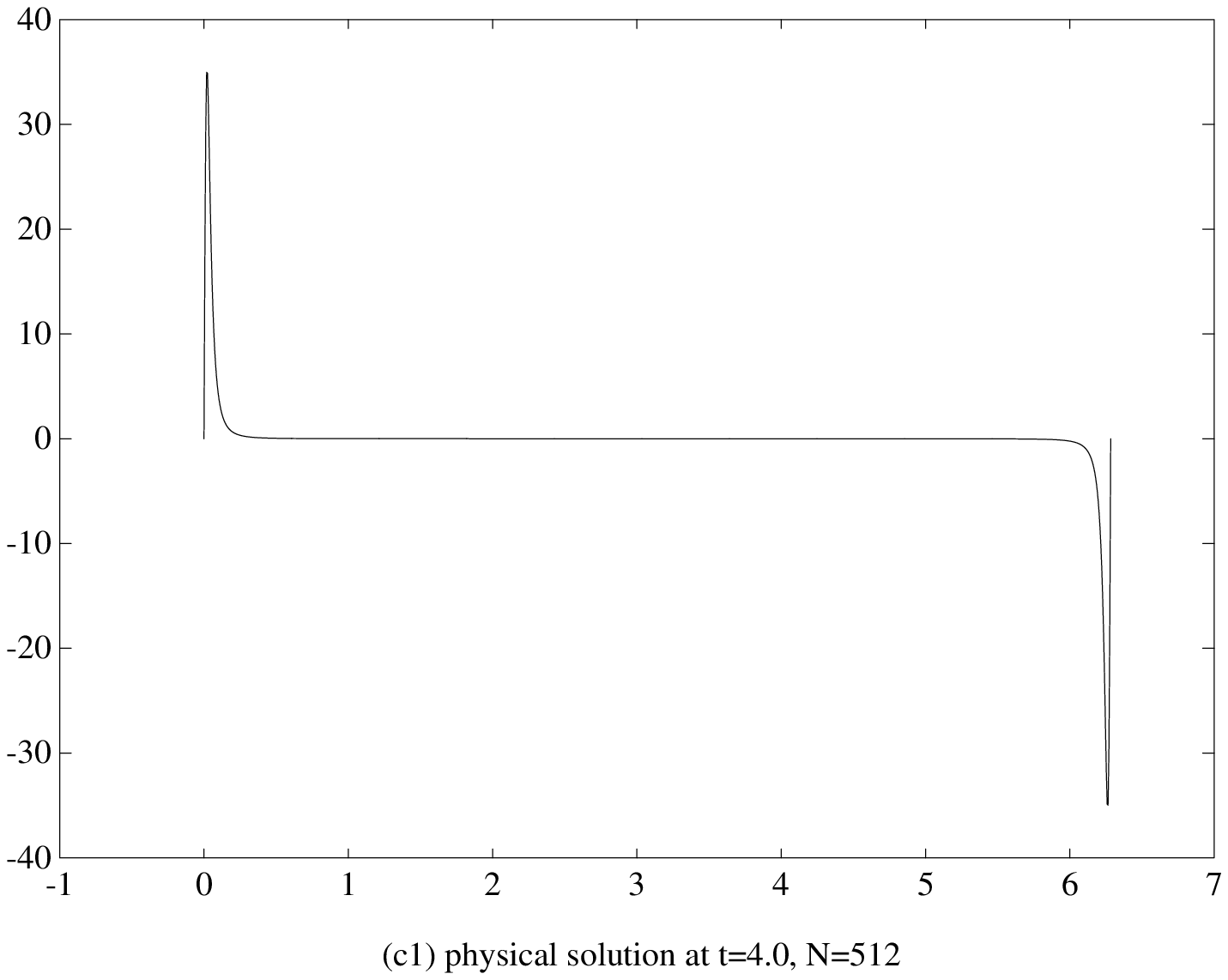}\\
\includegraphics[scale=0.3]{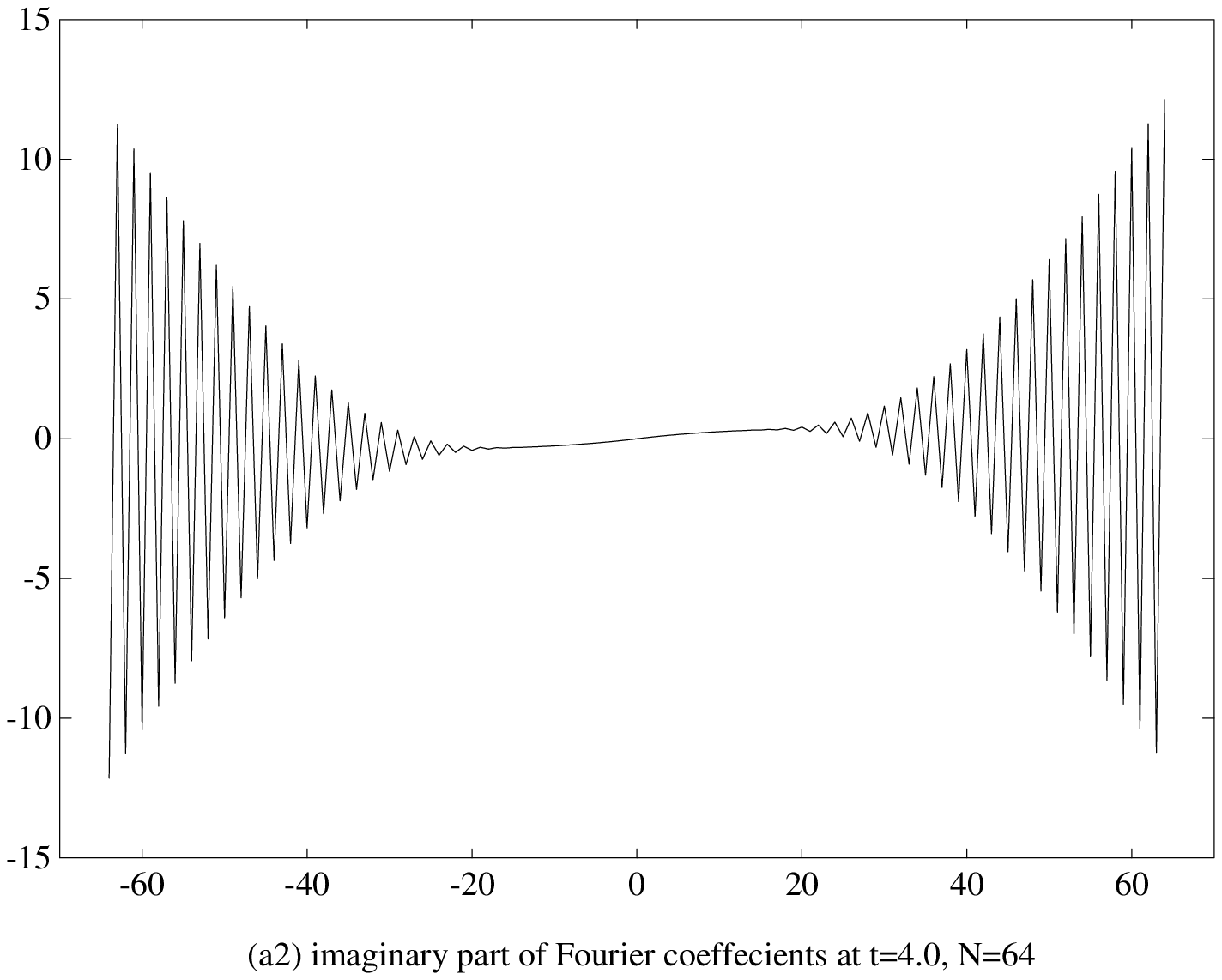}& \includegraphics[scale=0.3]{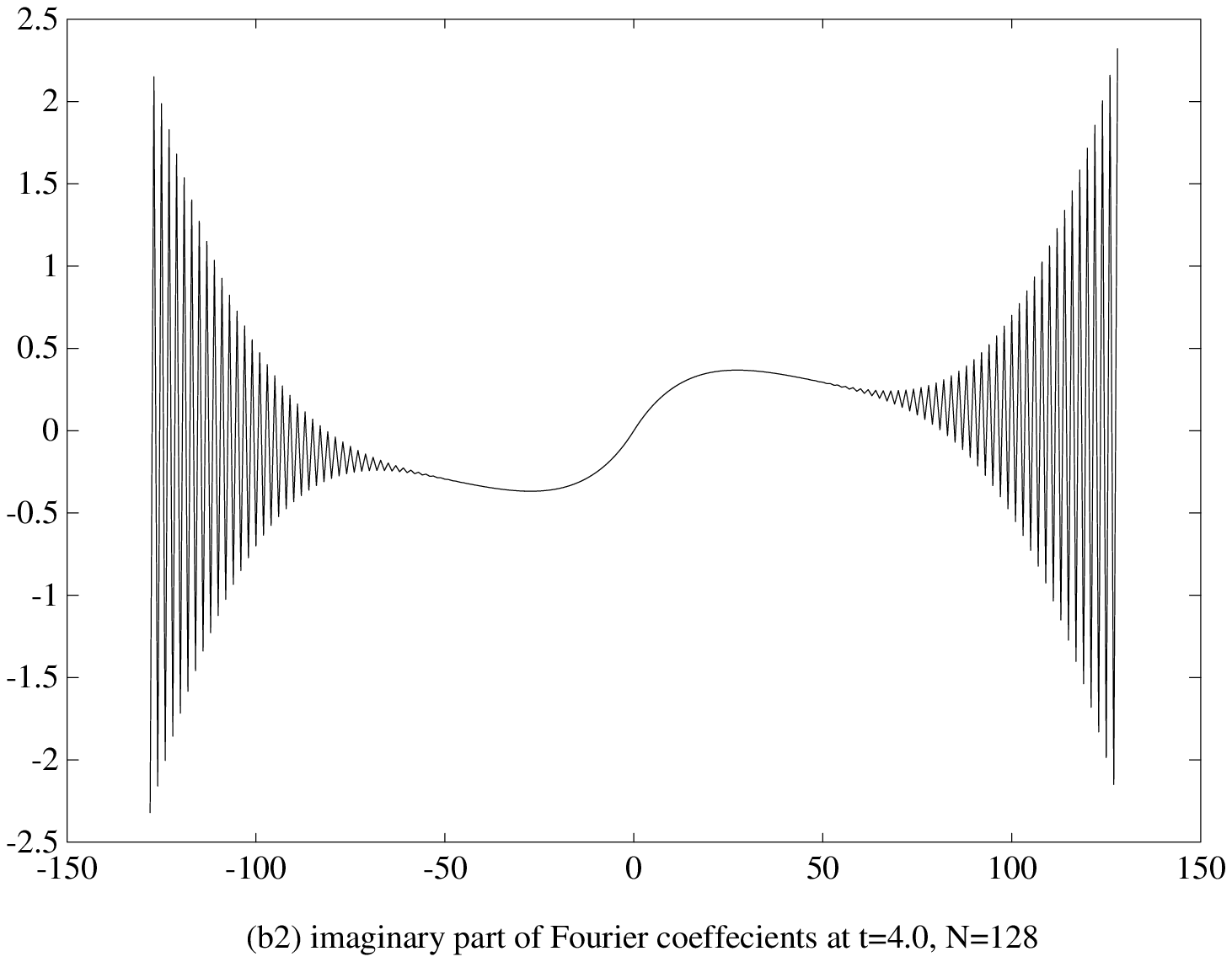}& \includegraphics[scale=0.3]{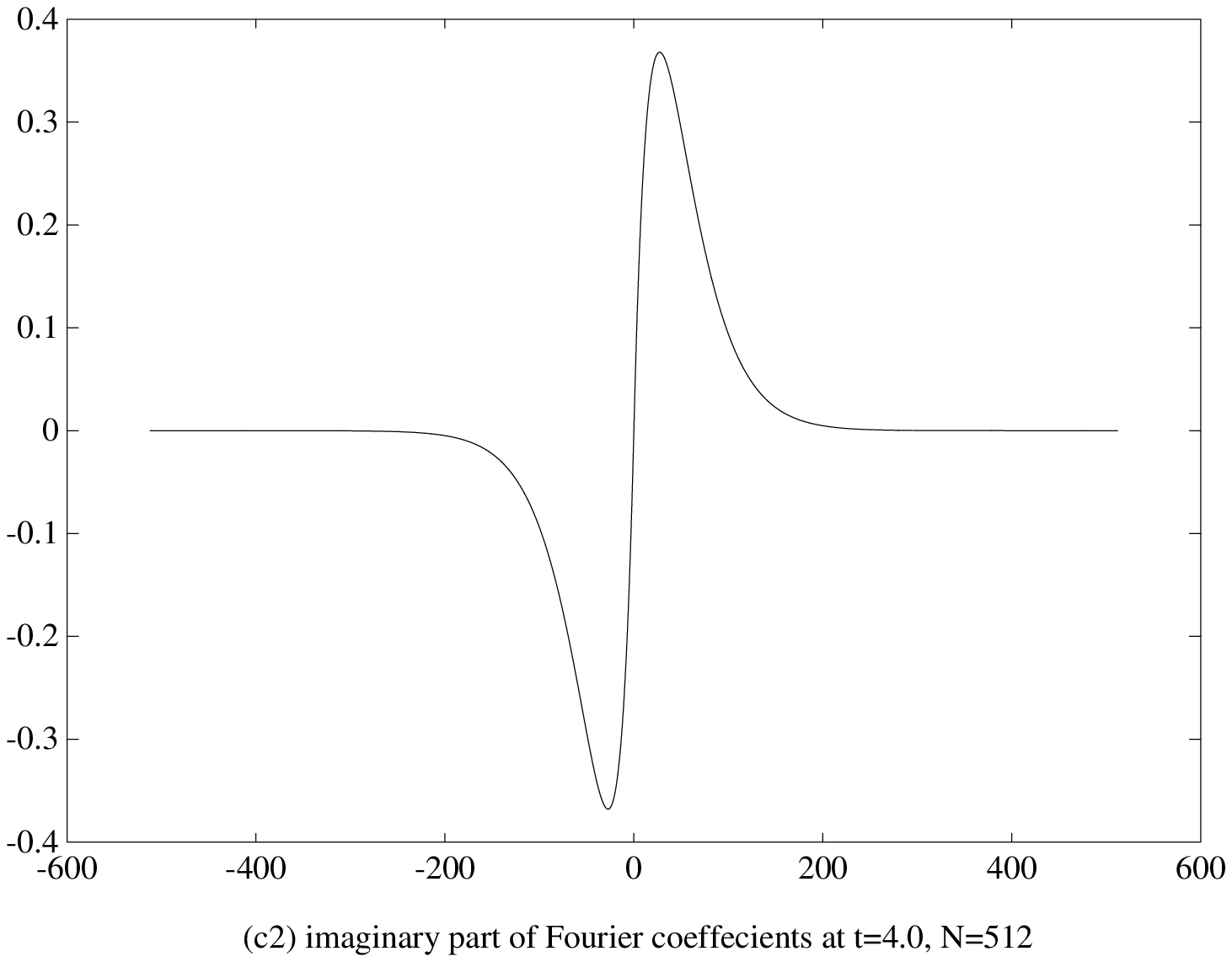}
\end{tabular}
\end{center}
\caption{ $\Im\hat{u}_{k}(t=0.4)$ of $u_N(\cdot,t)$ with $N=64$, $N=128$ and $N=512$ modes.}\label{fig:c} 
\end{figure}

\subsection{De-aliasing: the $2/3$ smoothing method and strong stability}
One way to regain the stability of the pseudo-spectral Fourier method in (\ref{eq:pspsin}) is to set $\widehat{u}_N(t)\equiv 0$ which  prevents the  growth of $b_N(t)\equiv 0$.  Thus, removing the last mode stabilize the pseudospectral method in the special case of $q(x)=\sin(x)$. The hyperbolic equation (\ref{eq:pde}) with a general $q(x)$ follows along the lines of \cite{Tad87}. We return to the aliasing term (\ref{eq:semi_aliasing}),
\begin{eqnarray}\label{eq:alia}
\lefteqn{\int \uN(x,t)\frac{\partial}{\partial x}{\AN}[q(\cdot)\uN](x,t)dx} \\
& & \qquad \qquad   = 2\pi i\sum_{|j|,|k|\leq N}
\overline{\widehat{u}}_j(t) \widehat{u}_{k}(t) \,(j-k)\cdot\sum_{{\aj\neq 0}}\widehat{q}\left({j-k+\aj(2N+1)}\right). \nonumber
\end{eqnarray}
As noted before, the ${\mathcal O}(N)$ growth of the high  Fourier modes, $|\widehat{u}_j(t)|, \ |j|\sim N$,  need not be small due to lack of apriori smoothness. We can circumvent this difficulty if we remove these modes by setting a fixed portion of the spectrum to be zero. For example, assume that we truncate the last $1/3$ of the modes of $\uN$ (any other fixed fraction of $N$ will do). To this end, we use a \emph{smoothing operator} $\SR$ which is activated only on the first $\tthirds N$ modes while removing the top $1/3 N$ of the modes. We up with the so-called $2/3$ pseudo-spectral Fourier method, 
\begin{subequations}\label{eqs:pspsmoo}
\begin{equation}\label{eq:pspsmoo}
 \frac{\partial}{\partial t}\uN(x,t)+\frac{\partial}{\partial x}\IN[q(\cdot)\SR\uN](x,t)=0, \qquad  \SR\uN:= \sum_{|k|\leq \tthirds N}\sigma_k\widehat{u}_k(t)e^{ikx};
\end{equation}
To retain spectral accuracy, the smoothing factors $\sigma_k \in (0,1]$ do not change  a fixed portion of the lower spectrum 
\begin{equation}\label{eq:pspsmoob}
\sigma_k \left\{\begin{array}{ll} \equiv 1, & |k|\leq \frac{1}{3}N \\ \\ \in (0,1], & \frac{1}{3}N <|k| <\tthirds N.\end{array}\right.
\end{equation}
\end{subequations}
The $L^2$-stability of the $2/3$ method follows along the lines of the spectral stability in (\ref{eq:spstab}); integrating (\ref{eq:pspsmoo}) against $\SR\uN$ we find
\begin{subequations}\label{eqs:last}
\begin{equation}\label{eq:lasta}
\left|\int (\SR\uN)\frac{\partial}{\partial x}\SN[q(\cdot)\SR\uN]dx\right| \leq \hf q'_\infty \|\SR\uN(\cdot,t)\|^2_{L^2}, \qquad q'_\infty= \max_x|q'(x)|. 
\end{equation}
The  aliasing contribution in the $2/3$ method is bounded (and in fact negligible for $q\in C^s, s>1$):  following (\ref{eq:alia}) 
\[
\int (\SR\uN)\frac{\partial}{\partial x} {A_N}[q(\cdot)(\SR\uN)]dx 
 = 2\pi i\!\!\!\sum_{|j|\leq \tthirds N}\sum_{|k|\leq \tthirds N}\sigma_k\sigma_j
\overline{\widehat{u}}_j(t) \widehat{u}_{k}(t)\, (j-k)\cdot\sum_{{\aj\neq 0}}\widehat{q}\overbrace{\left({j-k+\aj(2N+1)}\right)}^{|{j-k+\aj(2N+1)}|\geq \tthirds N};
\]
observe that the  terms involved in the inner summation on the right  are now restricted to high wavenumbers, $|j-k+\aj(2N+1)|\geq \tthirds N$ so that $|\widehat{q}\left({j-k+\aj(2N+1)}\right)| \lesssim \|q\|_{C^r}N^{-r}$.  Hence
\begin{equation}\label{eq:lastb}
\left| \int (\SR\uN)\frac{\partial}{\partial x} {A_N}[q(\cdot)(\SR\uN)]dx \right|
 \lesssim  \|q\|_{C^r}N^{1-r} \times \|\SR\uN\|^2, \qquad r\geq 1.
\end{equation}
\end{subequations}
 Combining the last two inequalities (\ref{eq:lasta}) and (\ref{eq:lastb}) with $r=1$, we arrive at
\[
\hf\frac{d}{dt}\int(\SR\uN)(x,t)\uN(x,t)dx = -\int (\SR\uN)\ddx (\SN+A_N)[q(\cdot)(\SR\uN)]dx   \leq C q'_\infty \|\SR\uN(\cdot,t)\|^2_{L^2}.
\]
Thus, by activating the smoothing operator we removed aliasing errors and  the resulting $2/3$ de-aliased pseudo-spectral method (\ref{eqs:pspsmoo}) regained the  \emph{weighted} $L^2$-stability
\[
\|\uN(\cdot,t)\|^2_{L^2_{\SR}} \leq e^{2C q'_\infty t}\|\uN(\cdot,0)\|^2_{L^2_{\SR}}, \qquad
\|w(\cdot,t)\|^2_{L^2_{\SR}}:=\int (\SR w)(x,t)w(x,t)dx \equiv 2\pi\!\!\!\sum_{|k|\leq \tthirds N} \sigma_k|\widehat{w}_k(t)|^2.
\]
 The corresponding error equation for $e_N:=\SR\uN-\SR u$ reads (we skip the details)
\[
\ddt e_N + \ddx \big(\SR[q e_N]\big) = -\ddx \SR \big[q(x)(I-\SR)[u](x,t)\big],
\]
and the spectral convergence rate, (\ref{eq:rate}), follows: for $s>1$ there exists a constant, $C=C_s$ such that 
\[
\|\uN-u\|_{L_{\SR}^2} \lesssim e^{C_s q'_\infty t}\left(N^{-s}\|u(\cdot,0)\|_{H^s}  + N^{1-s}\max_{\tau\leq t}\|u(\cdot,\tau)\|_{H^s}\right), \quad s>1.
\]

\subsection{Spectral accuracy and propagation of discontinuities}\label{sec:disc}
Hyperbolic equations propagates $H^s$ regularity: 
$ \|u(\cdot,t)\|_{H^s}  \lesssim  {e^{C_s t}} \|u(\cdot,0)\|_{H^s} <\infty$. Thus, the convergence statement in (\ref{eq:rate}) implies spectral convergence of the spectral Fourier method and $2/3$ Fourier method
for $H^s$-smooth initial data. However, when the initial data is \emph{piecewise smooth}, the exact solution propagates discontinuities along characteristics, and  the (pseudo-)spectral approximations of jump discontinuities in $u(\cdot,t)$ produces spurious Gibbs oscillations, \cite{Tad07}.
Nevertheless, thanks to the $H^s$-stability of the spectral Fourier method and the $2/3$ pseudo-spectral methods,
$\|\uN(\cdot,t)\|_{\dot{H}^s} \lesssim e^{{C_s}|q|_\infty t}\|\uN(\cdot,0)\|_{\dot{H}^s}$,  measured in the \emph{weak} topology of $s<0$, the (pseudo-)spectral approximations still propagate accurate information of the smooth portions of the exact solution. This is realized in terms of the convergence rate (we skip the details)
\[
\|\uN-u\|_{H^{r}}\lesssim e^{C_s q'_\infty t}\left(N^{r-s}\|u(\cdot,0)\|_{H^{s}}  + N^{1+r-s}\max_{\tau\leq t}\|u(\cdot,\tau)\|_{H^{s}}\right), \qquad r < s-1 <-1.
\]
It follows that one can pre- and post-process $\uN(\cdot,t)$ to recover the \emph{pointvalues} of $u(\cdot,t)$ within spectral accuracy, away from the singular set of the solution, \cite{MMO78,ML78,AGT86}. The point to note here is that even the Fourier projections of the exact solution, $\SN u(\cdot,t)$ and $\IN u(\cdot,t)$ are at most first-order accurate due to Gibbs oscillations; the post-processing of the computed $\uN$ is realized by its smoothing using a proper $\sigma$-mollifier (\ref{eq:pspsmoob}) (or see (\ref{eq:mollifier}) below), which does both --- retains the stability and recovers the spectrally accurate resolution content of  the Fourier method.   

\section{Fourier method for Burgers equation: convergence for smooth solutions}\label{sec:burgers}
We now turn our attention to spectral and pseudo-spectral approximations of nonlinear problems.
Their spectral accuracy often make them the method of choice for simulations where the highest resolution is sought for a given number of degrees of freedom. 
We begin with the prototypical example for quadratic nonlinearities, the inviscid Bugers' equation, 
\begin{equation}\label{eq:burgers}
\frac{\partial}{\partial t}u(x,t) + \hf\frac{\partial}{\partial x}u^2(x,t)=0, \qquad x\in {\mathbb T}([0,2\pi)),
\end{equation}
subject to $\tpi$-periodic boundary conditions and prescribed 
initial conditions, $u(x,0)$. In this section we show that as long as the solution of Burgers equation remains smooth for a time interval $t \leq T_c$, the spectral and $2/3$ de-aliased pseudo-spectral approximations  converge to the exact solution with spectral accuracy. 

\subsection{The spectral approximation of Burgers equation}\label{sec:sp-quad-stability}
The spectral approximation of (\ref{eq:burgers}), $\uN(x,t)=\sum \widehat{u}_k(t)e^{ikx}$, is governed by,
\begin{equation}\label{eq:spburg}
\frac{\partial}{\partial t}\uN(x,t)+\frac{1}{2}\frac{\partial}{\partial x}\Big( \SN\left[\uN^2\right](x,t) \Big)=0, \qquad \pil\leq x \leq \pir.
\end{equation}
The evaluation of the quadratic term  on the right is carried out using convolution and (\ref{eq:spburg}) amounts to a nonlinear system of $(2N+1)$ ODEs for $\widehat{\mathbf u}(t)=(\widehat{u}_{-N}(t),\ldots, \widehat{u}_N(t))^\top$.

\begin{theorem}[{\bf Spectral convergence for smooth solutions of Burgers' equations}]\label{thm:spsmooth} Assume that  for $0<t\leq T_c$, the solution of the Burgers equation \eqref{eq:pde} is smooth, $u(\cdot,t)\in L^\infty\left([0,T_c], C^{1+\alpha}(\pil,\pir]\right)$. Then, the spectral method \eqref{eq:spburg}  converges in 
$L^\infty\left([0,T_c], L^2(0,2\pi]\right)$, 
\[
\|\uN(\cdot,t)-u(\cdot,t)\|_{L^2} \rightarrow 0, \qquad  0\leq t \leq T_c.
\]
Moreover, the following spectral convergence rate estimate holds for all $s>\threehf$,
\[
\|\uN(\cdot,t)-u(\cdot,t)\|^2_{L^2} \lesssim e^{{\displaystyle \int_0^t \!\!|u_x(\cdot,\tau)|_\infty d\tau}} \left(N^{-2s}\|u(\cdot,0)\|^2_{H^s} + N^{\threehf-s}\max_{\tau\leq t}\|u(\cdot,\tau)\|_{H^s}\right)\!, \  s>\threehf.
\] 
\end{theorem}

\begin{proof}
We rewrite  the spectral approximation (\ref{eq:burgers})  in the form
\[
\ddt \uN + \ddx \frac{\uN^2}2=\frac12\ddx(I-\SN)[\uN^2].
\]
The corresponding energy equation reads
\begin{equation}\label{eq:spener}
\ddt  \frac{\uN^2}{2} +\ddx \frac{\uN^3}{6}= \frac{\uN}{2}\ddx(I-\SN)[\uN^2].
\end{equation}
Integration yields the energy balance
\[
\frac12\dt\int \uN^2(x,t)dx = \frac12\int \uN\del_x(I-\SN)[\uN^2]dx =: {\mathcal I}_1.
\]
The  term on the right vanishes by  orthogonality, $\displaystyle {\mathcal I}_1=-\frac12\int\frac{\partial \uN}{\partial x} (I-\SN)[\uN^2]dx =0$, and hence the solution is $L^2$-conservative, 
\begin{equation}\label{eq:spL2consv}
\|\uN(\cdot,t)\|_{L^2}= \|\uN(\cdot,0)\|_{L^2}.
\end{equation}

\medskip
Next, we integrate $ (\uN-u)^2\equiv{|\uN|^2}-{|u|^2} - 2u(\uN-u)$: after discarding all terms which are in divergence form, we are left with
\begin{eqnarray*}
\frac12\dt\int  (\uN-u)^2dx&=& \dt\int \Bigg(\frac{|\uN|^2}2-\frac{|u|^2}2- u(\uN-u)\Bigg)dx\\
&= &\frac12\int\uN\del_x(I-\SN)[\uN^2]dx -\int \del_t(u(\uN-u))dx =:{\mathcal I}_1+{\mathcal  I}_2.
\end{eqnarray*}
Recall that ${\mathcal I}_1$ vanishes. As for  the second term ${\mathcal I}_2$, we decompose it into two terms, 
\begin{equation*}
{\mathcal I}_2=\int \del_t\big(u(\uN-u)\big)dx \equiv \int u_t (\uN-u)dx+\int  u (\del_t  \uN-\del_t u)dx,
\end{equation*}
and using (\ref{eq:burgers}), (\ref{eq:spburg}) and (\ref{eq:spener}) to convert time derivatives to spatial ones, we find
\begin{eqnarray*}
{\mathcal I}_2&= &-\int u  u_x (\uN-u)dx - \int  u \del_{x}\!\!\left(\frac{\uN^2}2-\frac{u^2}2\right)dx +\frac12\int  u\del_x (I-\SN) [\uN^2]dx\\
&=& -\int u  u_x (\uN-u)dx + \int  u_x\left(\frac{\uN^2}2- \frac{u^2}2\right)dx -\frac12\int  u_x(I-\SN) [\uN^2]dx\\
&=&\int u_x \left( \frac{\uN^2}2-\frac{u^2}2-u (\uN-u)\right)dx-\frac12\int   u_x (I-\SN)[\uN^2]dx.
\end{eqnarray*}
Eventually, we end up with
\begin{subequations}\label{eqs:gron}
\begin{equation}\label{eq:grona}
\frac12\frac{d}{dt}\int |\uN(x,t)-u(x,t)|^2dx \leq \frac{|u_x(\cdot,t)|_{L^\infty}}{2} \int |\uN(x,t)-u(x,t)|^2dx-\frac12 e_N,
\end{equation}
where the error term, $e_N$, is given by
\begin{equation}\label{eq:gronb}
e_N:=\int \uN^2(I-\SN)[u_x]  dx
\end{equation}
\end{subequations}
Observe that under the hypothesis $u_x\in L_t^\infty C^{0,\alpha}_x$, and hence by Jackson's bound \cite{DL93} and the $L^2$-bound (\ref{eq:spL2consv}) one has 
\[
 |e_N(t)|\lesssim \max_x|(I-\SN)[u_x(x,t)]|\cdot\|\uN\|^2_{L^2} \lesssim \frac{\ln N}{N^\alpha}\|\uN(\cdot,0)\|^2_{L^2} \rightarrow  0.
\]
With (\ref{eqs:gron}) one obtains,
\begin{eqnarray*}
\int |\uN(x,t)-u(x,t)|^2dx  & \le & e^{{U'_\infty(t;0)}} \int |\uN(x,0)-u(x,0)|^2dx\\
&  + & \int_0^t e^{{U'_\infty(t;\tau)}} |e_N(\tau)|d\tau, \qquad {U'_\infty(t;\tau):=\int_{s=\tau}^t|u_x(\cdot,s)|_\infty ds}.
\end{eqnarray*}
and convergence follows. Moreover, with $\uN(\cdot,0)=\SN u(\cdot,0)$ we end up with spectral convergence rate estimate
\begin{eqnarray*}
\lefteqn{\int |\uN(x,t)-u(x,t)|^2dx} \\
& & \lesssim e^{\displaystyle {\int_0^t |u_x(\cdot,\tau)|_\infty d\tau}} \left(N^{-2s}\|u(\cdot,0)\|^2_{H^s}+ N^{\threehf-s}\max_{\tau\leq t}\|u(\cdot,\tau)\|_{H^s}\right), \qquad s> \threehf.
\end{eqnarray*}
\end{proof}

\subsection{The $2/3$ de-aliasing pseudo-spectral approximation of Burgers equation}\label{sec:ps-quad-stability}

Convolutions  can be avoided using the pseudo-spectral Fourier method, \cite{KO72,GO77},
\[
\frac{\partial}{\partial t}\uN(x,t)+\frac{1}{2}\frac{\partial}{\partial x}\Big( \IN\left[\uN^2\right](x,t) \Big)=0, \qquad x\in {\mathbb T}([0,2\pi)).
\]
 Observe that (\ref{eq:burgers}) is satisfied \emph{exactly} at the gridpoints $x_\nu$,
\[
\frac{d}{dt}\uN(x_\nu,t)+\frac{1}{2}\frac{\partial}{\partial x}\Big( \IN\left[\uN^2\right](x,t) \Big)_{\big|x=x_\nu}=0, \qquad \nu=0, 1, \ldots, 2N.
\]
The resulting  system of $(2N+1)$ nonlinear equations for ${\bf u}(t)=(u(x_0,t),\ldots,u(x_{2N},t))^\top$  can be then integrated in time by standard ODE solvers. 
 The pseudo-spectral approximation introduces aliasing errors. To eliminate these errors, we consider the $2/3$ de-aliasing Fourier method, consult (\ref{eq:pspsmoo})
\begin{subequations}\label{eqs:fourier}
\begin{equation}\label{eq:fourier}
\frac{\partial}{\partial t}\uN(x,t)+\frac{1}{2}\frac{\partial}{\partial x}\Big( \IN\left[(\SR\uN)^2\right](x,t) \Big)=0, \quad   x\in {\mathbb T}([0,2\pi)),
\end{equation}
where $\SR\uN$ denotes a \emph{smoothing} operator of the form
\begin{equation}\label{eq:fourierb}
 \SR\uN:= \sum_{|k|\leq \tthirds N}\sigma_k\widehat{u}_k(t) e^{ikx}, \qquad \widehat{u}_k(t)=\frac{h}{\tpi}\sum_{\nu=0}^{2N}u_N(x_\nu,t)e^{-ikx_\nu}.
\end{equation}
The smoothing operator $\SR$ is dictated by the smoothing factors, $\{\sigma_k\}_{|k|\leq \tthirds N}$,  which truncates  modes with wavenumbers $|k|>\tthirds N$ while leaving a fixed portion --- say, the first 1/3 of the spectrum, viscous-free.
This is the same smoothing operator $\SR\uN$ we considered already in the \emph{linear} $2/3$ method (\ref{eq:pspsmoo}).
In typical cases, one may employ a smoothing mollifier, $\sigma(\cdot) \in C^\infty(0,1)$,  setting
\begin{equation}\label{eq:mollifier}
\sigma_k= \sigma\left(\frac{|k|}{N}\right), \quad \sigma(\xi)   \left\{\begin{array}{ll} \equiv 1, & \xi\leq \frac{1}{3},  \\ \\
\in (0,1), & \frac{1}{3} < \xi < \tthirds, \\ \\  
\equiv 0, & \frac{2}{3} \leq \xi \leq 1.\end{array}\right. 
\end{equation} 
\end{subequations}
This is the $2/3$ de-aliasing Fourier method  which is often advocated for spectral computations,in particular those involving  quadratic nonlinearities, \cite{HL07,OHFS10,Kerr93,Kerr05}.  

\medskip
 In what sense does the $2/3$ method remove aliasing errors? to make precise the de-aliasing aspect of (\ref{eqs:fourier}),  consider the  $2/3$ truncated solution $\um:=\SR\uN$. Here we emphasize that we are dealing with the smoothed solution, $\um$,  of degree $m:= \tthirds N$. Observing that  truncation  commute with differentiation, we find 
\begin{equation}\label{eq:jN}
\frac{\partial }{\partial t}\um(x,t)+\frac{1}{2}\frac{\partial}{\partial x}\SR\left(\IN[\um^2]\right)(x,t)=0, \qquad deg(\um)= m:=\tthirds N.
\end{equation}
We now come to the key point behind the removal of aliasing in quadratic nonlinearities: 
since $\widehat{u}_m(k)=0$ for $|k|> \tthirds N$ then $\widehat{\um^2}(k)=0$ for $|k|> \frac{4}{3}N$ hence $\widehat{\um^2}({k+\aj(2N+1)})=0$ for $|k|\leq \tthirds N, \aj\neq 0$; consequently, since the smoothing operator $\mathcal{S}$ acts only on the first $\tthirds N$ mode,  $\mathcal{S}\left({\AN}\um^2\right)\equiv 0$,
and we conclude
\[
\SR \big(\IN [\um^2]\big)(x,\cdot) \equiv \SR \big(\left(\SN+{\AN}\right)[\um^2]\big)(x,\cdot)
= \SR \left(\SN[\um^2]\right)(x,\cdot) \equiv \SR \um^2(x,\cdot).
\]
We summarize by stating the following.
\begin{corollary}\label{cor:special}
 Consider the $2/3$ de-aliasing Fourier method \eqref{eqs:fourier}  then its $2/3$ smoothed solution,  $\um:=\SR\uN$, satisfies
\begin{equation}\label{eq:23}
\frac{\partial }{\partial t}\um(x,t)+\frac{1}{2}\frac{\partial}{\partial x}\SR[\um^2](x,t)=0, \qquad 
\SR w= \sum_{|k|\leq m} \sigma_k\widehat{w}_k e^{ikx}, \ m= \tthirds N.
\end{equation}
\end{corollary}

\noindent
Thus, by  truncating the top $1/3$ of the modes, we de-aliased the Fourier method, (\ref{eq:fourier}), in the sense that (\ref{eq:23}) does \emph{not} involve any aliasing errors: only truncation errors, $(I-\SR)[\um^2]$ are involved. Indeed, the formulation  of $2/3$ method in (\ref{eq:23}) resembles the $m$-mode spectral method (\ref{eq:spburg}). The only difference is due to the fact that unless $\sigma_k \equiv 1$,  the  smoothing operator ${\mathcal S}$ is not a projection\footnote{When $\sigma_k \equiv 1$, then ${\mathcal S}=P_{\tthirds N}$ and the $2/3$ method coincides with the spectral Fourier method (\ref{eq:spburg})
 with $ m= \tthirds N$ modes,
\[
\frac{\partial}{\partial t}\um(x,t)+\frac{1}{2}\frac{\partial}{\partial x}P_{m}[\um^2](x,t)=0,  \qquad  \ |k|\leq \tthirds N. 
\]}

\medskip
The following theorem shows that as long as the Burgers solution remains smooth, the $2/3$ de-aliasing Fourier method is stable and enjoys spectral convergence.
 
\begin{theorem}[{\bf Spectral convergence of the $2/3$ method for smooth solutions}]\label{thm:pspsmooth}  Assume that  for $0<t\leq T_c$, the solution of the Burgers equation \eqref{eq:pde} is smooth, $u(\cdot,t)\in L^\infty\left([0,T_c], C^{1+\alpha}(\pil,\pir]\right)$. Then, the $2/3$ de-aliasing method \eqref{eqs:fourier}  converges in $L^\infty\left([0,T_c], L^2(0,2\pi]\right)$, 
\[
\|\um(\cdot,t)-u(\cdot,t)\|_{L^2} \rightarrow 0, \qquad  0\leq t \leq T_c,
\]
and the following spectral convergence rate estimate holds
\[
\|\uN(\cdot,t)-u(\cdot,t)\|^2_{L^2} \lesssim e^{\displaystyle {\int_0^t |u_x(\cdot,\tau)|_\infty d\tau}} \!\!\!\left(N^{-2s}\|u(\cdot,0)\|^2_{H^s}+ N^{\threehf-s}\max_{\tau\leq t}\|u(\cdot,\tau)\|_{H^s}\right), \  s> \threehf.
\]
\end{theorem}

\begin{proof} We start with (\ref{eq:23})
\[
\frac{\partial }{\partial t}\um(x,t)+\frac{1}{2}\frac{\partial}{\partial x}\Big(\SR[\um^2](x,t)\Big)=0.
\]
Since  $\SR$ need not be a projection, there is no $L^2$-energy conservation for the $2/3$ smoothed solution $\um$. Instead,
 we integrate    against $\uN$ to find that the  corresponding energy balance reads 
\begin{eqnarray*}
\hf \dt\int\uN(x,t)\um(x,t)dx & = & -\frac12\int \uN\frac{\partial}{\partial x}\SR[\um^2](x,t)dx  \\
&   = & \quad \frac12\int \ddx(\SR\uN) \um^2(x,t)dx =\frac{1}{6}\int \ddx \um^3 dx =0,
\end{eqnarray*}
and hence the solution conserve the weighted $L^2_{\SR}$-norm, 
\begin{equation}\label{eq:pspl2cons}
 \|\um(\cdot,t)\|^2_{L^2_\SR}= \|\uN(\cdot,0)\|^2_{L^2_\SR}, \qquad \|\uN(\cdot,t)\|^2_{L^2_\SR}:= \int (\SR\uN) \uN dx = 2\pi \!\!\!\sum_{|k|\leq \tthirds N} \sigma_k|\widehat{u}_k(t)|^2.
\end{equation}

We proceed along the lines of the spectral proof in theorem \ref{thm:spsmooth}, integrating 
$|\um-u|^2\equiv |\um|^2-|u|^2 -2u(\um-u)$:  after discarding all terms which are in divergence form, we are left with
\begin{eqnarray*}
\frac12\dt\int  (\um-u)^2dx&=& \dt\int \Bigg(\frac{|\um|^2}2-\frac{|u|^2}2- u(\um-u)\Bigg)dx\\
&= &\frac12\dt \int |\um|^2dx -\int \del_t(u(\um-u))dx =:{\mathcal I}_1+{\mathcal  I}_2.
\end{eqnarray*}
Unlike the $L^2$ conservation of the spectral solution $\uN$, consult (\ref{eq:spL2consv}), there is no $L^2$-energy conservation for the $2/3$ smoothed solution $\um$ and we therefore leave ${\mathcal I}_1$ is left as perfect time derivative. As for  the second term 
\begin{equation*}
{\mathcal I}_2=\int \del_t\big(u(\um-u)\big)dx \equiv \int \del_t u (\um-u)dx+\int  u (\del_t  \um-\del_t u)dx,
\end{equation*}
we reproduce the same steps we had  in the spectral case:  using (\ref{eq:burgers}) and (\ref{eq:23})  to convert time derivatives to spatial ones, we find
\begin{eqnarray*}
{\mathcal I}_2&= &-\int u  u_x (\um-u)dx - \int  u \del_x\left(\frac{\um^2}2- \frac{u^2}2\right)dx +\int  u\del_x (I-\SR) [\um^2]dx\\
&=& -\int u  u_x (\um-u)dx + \int  u_x\left(\frac{\um^2}2- \frac{u^2}2\right)dx -\frac12\int  u_x(I-\SR) [\um^2]dx\\
&=&\int u_x \left( \frac{\um^2}2-\frac{u^2}2-u (\um-u)\right)dx-\frac12\int   u_x (I-\SR)[\um^2]dx.
\end{eqnarray*}
Eventually, we end up with
\[
\frac12\frac{d}{dt}\int |\um(x,t)-u(x,t)|^2dx \leq    \frac{|u_x|_\infty}{2} \int |\um(x,t)-u(x,t)|^2dx+\frac12 e_N(t) + \frac12 \dt\int |\um(x,t)|^2dx,
\]
where the error term, $e_N$ is given by $\displaystyle e_N(t):=-\int \um^2 (I-\SR)[u_x] dx$.
Integrating  in time we find 
\begin{eqnarray}
\hspace*{1.4cm}  \lefteqn {\int_x|\um(x,t)-u(x,t)|^2dx -\int_x|\um(x,0)-u(x,0)|^2dx} \label{eq:pspgron}  \\
& & \ \ \ \ \ \leq   
  |u_x|_\infty \int_{\tau=0}^t \int |\um(x,t)-u(x,t)|^2dxd\tau + \int_0^t e_N(\tau)d\tau +  f_N(t), \nonumber
\end{eqnarray}
with the additional error term, $f_N(t)$, given by
\[
f_N(t) :=  \int |\um(x,t)|^2dx - \int |\um(x,0)|^2dx.
\]
The error term $e_N(t)$ can be estimated as before: observe that under the hypothesis $u_x\in L_t^\infty C^{0,\alpha}_x$, one has 
\begin{equation}\label{eq:xyz}
 |e_N(t)|\lesssim \max_x|(I-\SR)[u_x(x,t)]|\cdot\|\um\|^2_{L^2} \lesssim \frac{\ln N}{N^\alpha}\|\uN(\cdot,0)\|^2_{L^2} \rightarrow  0.
\end{equation}

To address the new error term, $f_N(t)$, we observe by the $L^2_{\SR}$-energy conservation (\ref{eq:pspl2cons}),
\begin{eqnarray*}
\int |\um(x,t)|^2dx & = & \sum_{|k|\leq \tthirds N} \sigma_k^2 |\widehat{u}_k(t)|^2 \leq \sum_{|k|\leq \tthirds N} \sigma_k 
|\widehat{u}_k(t)|^2  =  \sum_{|k|\leq \tthirds N} \sigma_k |\widehat{u}_k(0)|^2  \\
 & = & \sum_{|k|\leq \tthirds N} \sigma^2_k|\widehat{u}_k(0)|^2 + (\sigma_k-\sigma^2_k)|\widehat{u}_k(0)|^2 \\
 &= & \int |\um(x,0)|^2 + \sum_{|k|\leq \tthirds  N} \left(\sigma_k-\sigma^2_k\right) |\widehat{u}_k(0)|^2.
\end{eqnarray*}
Since $\sigma_k\equiv 1$ for $|k|< N/3$, consult (\ref{eq:mollifier}), we conclude
\begin{eqnarray}\label{eq:abc}
f_N(t) &:= & \int |\um(x,t)|^2 -\int |\um(x,0)|^2 \\
& \leq & \sum_{\frac{1}{3}N \leq  |k|\leq \tthirds  N} \left(\sigma_k-\sigma^2_k\right) |\widehat{u}_k(0)|^2 \leq  \big\|\big(P_{\tthirds N}-P_{\frac{1}{3}N}\big)u(\cdot,0)\big\|_{L^2}^2 \rightarrow 0. \nonumber
\end{eqnarray}
With (\ref{eq:pspgron}), (\ref{eq:abc}) and (\ref{eq:xyz})  in place, one obtains an  estimate  on the error integrated in space-time
\[
\frac12 \dt E_m(t)  \leq  \frac{|u_x|_\infty}{2}E_m(t)  + \frac12\int_0^t e_N(\tau)d\tau + \frac12 f_N(t), \quad  E_m(t):=\int_0^t \int |\um(x,\tau)-u(x,\tau)|^2dxd\tau.
\]
Convergence follows  by Gronwall's inequality,
\begin{eqnarray*}
\lefteqn{\int |\um(x,t)-u(x,t)|^2dx} \\
& &    \lesssim   e^{\displaystyle{\int_0^t \!\!|u_x(\cdot,\tau)|_\infty d\tau}} 
 \!\!\Big(\|\um(\cdot,0)-u(\cdot,0)\|^2_{L^2}  \\
  & & \hspace*{3.2cm} \left. + \max_{x,  \tau\leq t} |(I-\SR)u_x(x,\tau)| + \big\|\big(P_{\tthirds N}-P_{\frac{1}{3}N}\big)u(\cdot,0)\big\|_{L^2}^2\right). 
\end{eqnarray*}
Moreover, with $\uN(\cdot,0)=\SN u(\cdot,0)$ we end up with spectral convergence rate estimate
\begin{eqnarray*}
\lefteqn{\int |\um(x,t)-u(x,t)|^2 dx} \\
& &  \lesssim e^{\displaystyle{\int_0^t |u_x(\cdot,\tau)|_\infty d\tau}} \left(N^{-2s}\|u(\cdot,0)\|^2_{H^s}+ N^{\threehf-s}\max_{\tau\leq t}\|u(\cdot,\tau)\|_{H^s}\right), \quad s>\threehf.
\end{eqnarray*}
\end{proof}

\section{Fourier method for Burgers equation: instability for weak solutions}\label{sec:2-3rd}
In this section we discuss the spectral and the $2/3$ de-aliased pseudo-spectral Fourier approximations of Burgers'  equation, (\ref{eq:burgers}),  after the  formation of shock discontinuities. We show that both methods are unstable after the critical time, $t>T_c$. 
Recall that the spectral method is a special case of the $2/3$ de-aliased method when we set the smoothing factors $\sigma_k\equiv 1$, see corollary \ref{cor:special}. It will therefore suffice to consider the $2/3$ de-aliasing pseudo-spectral Fourier method (\ref{eq:23}). We begin with its $L^2_{\SR}$-conservation (\ref{eq:pspl2cons}), which we express as 
\begin{equation}\label{eq:l2energy}
\|\SR^{1/2}\uN(\cdot,t)\|_{L^2}= \| \SR^{1/2}\uN(\cdot,0)\|_{L^2}, \qquad \SR^{1/2}\uN:=\sum_{|k|\leq m} \sqrt{\sigma_k}\,\widehat{u}_k(t).
\end{equation}
Since the quadratic energy associated with $\SR^{1/2}\uN$ is bounded, it follows that, after extracting a subsequence if necessary\footnote{Here and below we continue to label such subsequences as $\uN$.} that $\SR^{1/2}\uN(\cdot,t)$ and hence $\um=\SR\uN$ has a $L^2$-weak limit, $\ubar(x,t)$. But $\ubar$ \emph{cannot} be the physically relevant entropy solution of (\ref{eq:pde}). Our next result quantifies what can go wrong.
 
\begin{theorem}[{\bf The  $2/3$ method must admit spurious oscillations}]\label{thm:instab}
Let $T_c$ be the critical time of shock formation  in Burgers' equation \eqref{eq:burgers}. Let $\um=\SR\uN$ denote the smoothed $2/3$ de-aliasing Fourier method, \eqref{eqs:fourier}. Assume the $L^6$-bound, $\|\um(\cdot,t)\|_{L^6} \leq Const$ holds. 
Then,  for $t>T_c$, there exists a constant $c_0>0$ (independent of $N$) such that\footnote{$\|\um\|_{TV}$ denotes the total variation of $\um$.} 
\begin{equation}\label{eq:instab}
\max_{x}|\um(x,t)|\times\|\um(\cdot,t)\|^2_{TV} \geq c_0 \sqrt{m}.
\end{equation}
\end{theorem}

Theorem \ref{thm:instab} implies that either the  solution of the  $2/3$ de-aliasing Fourier method, $\um=\SR\uN$, grows unboundedly, 
\[
\lim_{N\rightarrow \infty} \|\um(\cdot,t)\|_{L^\infty} \longrightarrow \infty,
\]
or it has an unbounded total variation of order $\geq {\mathcal O}(\sqrt[4]{N})$. Each one of these scenarios implies that $\um$ contains spurious oscillations
which are noticeable \emph{throughout} the computational domain, in agreement with the numerical evidence observed in \cite{Tad89}. We note that this type of nonlinear instability applies to both, the $2/3$ method and in particular, the spectral Fourier method and  we refer in this context to the recent detailed study in \cite{RFNM11,PNFS13} and the refernces therein.

\begin{proof} 
We begin with (\ref{eq:23}) 
\begin{equation}\label{eq:23res}
\ddt\um(x,t)+\hf \ddx\um^2(x,t)=
\hf \frac{\partial}{\partial x}(\Id-\SR)[\um^2](x,t).
\end{equation}
Observe that the residual on the right tends  to zero in $H^{-1}$, 
\begin{eqnarray*}
\Big|\int\ddx\varphi(x)(\Id-\SR)[\um^2](x,t)dx \Big| & = &  
\Big|\int\Big((\Id-\SR)\ddx \varphi(x)\Big)\um^2(x,t)dx\Big| \\
& \leq & \|\um(\cdot,t)\|^2_{L^4}\times
\|(I-\SR)\varphi_x(\cdot)\|_{L^2}  \rightarrow 0, \qquad \forall \varphi \in H^1. 
\end{eqnarray*}

Next, we consider the $L^2$-energy balance associated with (\ref{eq:23res}).
Multiplication by $\um$ yields
\begin{equation}\label{eq:entropy}
\hf \ddt \um^2(x,t) +\frac{1}{3}\ddx \um^3(x,t)=
\hf \um(x,t) \ddx(\Id-\SR)[\um^2](x,t).
\end{equation}
We continue our argument by  claiming that if (\ref{eq:instab}) fails, then  the energy production on the right of (\ref{eq:entropy}) also tends weakly to zero in $H^{-1}$. To this end, we examine the weak form of the expression on the right which we rewrite as 
\[
\int\varphi(x)\um(x,t)\frac{\partial}{\partial x}(P_{2m}-\SR)[\um^2](x,t)dx =
\int (P_{2m}-\SR)\Big(\varphi(x)\um(x,t)\Big) \ddx\um^2(x,t)dx.
\]
It does not exceed
\begin{subequations}\label{sub:eqs}
\begin{eqnarray}
\lefteqn{\Big|\int\varphi(x)\um \ddx(\Id-\SR)[\um^2](x,t)dx\Big|} \nonumber \\
& & =  \Big|\int(P_{2m}-\SR)\Big(\varphi(x)\um(x,t)\Big)
\um(x,t)\frac{\partial}{\partial x}\um(x,t)dx\Big| \\
&  & \leq 
\big\|(P_{2m}-\SR)\big(\varphi(x)\um(x,t)\big)\big\|_{L^\infty}\times \|\um(\cdot,t)\|_{TV}\times\|\um(\cdot,t)\|_{L^\infty}. \nonumber
\end{eqnarray}
 To upper bound the first term we use standard decay estimate, $|\sigma_j\widehat{u}_N(j)(t)|\lesssim {\|\um(\cdot,t)\|_{TV}}/(1+|j|)$. Noting that  $P_{2m}-\SR$ annihilates the first $m/2$ modes, namely, the multipliers $\widehat{P_{2m}\!\!-\!\!\SR}(k)=0, \ |k|\leq m/2=N/3$, we find
\begin{align} 
\big\|(P_{2m}&-\SR)\Big(\varphi(x)\um(x,t)\Big)\big\|_{L^\infty} \nonumber \\
& \leq  \sum_{\frac{m}{2}  \leq |k|\leq  2m} (1-\sigma_k)\Big| \sum_{|j|\leq m}\widehat{\varphi}(k-j)\sigma_j\widehat{u}_N(j,t)\Big|  \nonumber \\
& \lesssim  \sum_{\frac{m}{2} \leq  |k|\leq  2m}  \sqrt{\sum_{|j|\leq m}
(1+|k-j|^2)|\widehat{\varphi}(k-j)|^2} \cdot \sqrt{\sum_{|j|\leq m} \frac{1}{(1+|k-j|^2)(1+|j|^2)}} \times \|\um(\cdot,t)\|_{TV} \\
&  \lesssim \|\varphi\|_{H^1}\|\um(\cdot,t)\|_{TV} \times \frac{1}{\sqrt{m}}. \nonumber
\end{align} 
\end{subequations}
The last two inequalities (\ref{sub:eqs}) give us,
\[
\Big|\int\varphi(x)\um(x,t) 
\frac{\partial}{\partial x}(\Id-\SR)[\um^2](x,t)dx\Big| \lesssim
\frac{1}{\sqrt{m}}\|\um(\cdot,t)\|^2_{TV}\times \|\um(\cdot,t)\|_{L^\infty} \times \|\varphi\|_{H^1}.
\]
We claim that (\ref{eq:instab}) holds by contradiction. If it fails, then we can choose a subsequence, $\umk$, such that
\[
\frac{1}{m_k}\|\umk(\cdot,t)\|^2_{TV}\times \|\umk(\cdot,t)\|_{L^\infty} \leq c_k, \qquad c_k \downarrow 0,
\]
and the energy production on the right of (\ref{eq:entropy}) vanishes in $H^{-1}$. By assumption $\um^r(\cdot,t) \in L^2$ for $r=1,2,3$ and the div-curl lemma, \cite{Mu78,Tar79,Tar87} applies: it follows that $\ubar$ is in fact a \emph{strong} $L^2$-limit, $\umk \rightarrow \ubar$.
Passing to the weak limit in (\ref{eq:23res})${}_{m_k}$  we have that $\ubar$ is weak solution of Burgers' equation (\ref{eq:burgers}),
\[
\frac{\partial}{\partial t}\ubar(x,t)+\frac{\partial}{\partial x}\Big(\frac{\ubar^2(x,t)}{2}\Big)=0.
\]
Moreover, passing to the weak limit in the energy balance (\ref{eq:entropy})${}_{m_k}$, we conclude that $\ubar$ satisfies the quadratic entropy \emph{equality}
\[
\frac{\partial}{\partial t}\Big(\frac{\ubar^2(x,t)}{2}\Big)+\frac{\partial}{\partial x}\Big(\frac{\ubar^3(x,t)}{3}\Big)=0.
\]
But,  due to the uniqueness enforced with by the single entropy -- in this case, the $L^2$ energy, \cite{Pan94}, there exists no  energy \emph{conservative} weak solution of Burgers equation (\ref{eq:burgers})  after the critical time of shock formation. 
\end{proof}

\begin{remark}\label{rem:smoo}
The same result of instability holds if we employ the pseudo-spectral Fourier method with a general smoothing operator beyond just the $2/3$ smoothing, namely $\SR\uN= \sum_{|k|\leq N} \sigma_k \hat{u}_ke^{ikx}$ and smoothing factors $\sigma_k$ decay too fast as $|k|\uparrow N$. 
\end{remark}

\section{Fourier method for Euler equations: convergence for smooth solutions}\label{sec:smoothEuler}
Convergence of the spectral and pseudo-spectral approximation for the Burgers equation made use of its quadratic flux, ${u^2}/2$. The same approach can be pursued for the Euler equations,
\begin{equation}\label{eq:euler}
\ddt \bu +\Lr\nablax (\bu\otimes \bu)=0, \qquad \bx \in {\mathbb T}^d,
\end{equation}
where  $\Lr:=Id- \nablax\Delta^{-1}\text{div}_\bx$ is the Leray projection into divergence free vector fields.

\subsection{Convergence of spectral Fourier approximation for Euler equations}
The spectral method for the Euler equations reads
\begin{equation}\label{eq:Eulerapp}
\ddt \buN +\Lr\nablax \SN(\buN\otimes \buN)=0.
\end{equation}
Convergence  for smooth solutions in this  case, is in fact even  simpler than in Burgers' equation. Observe that for any divergence free vectors fields, $\bv$ and $\bu$, the following identity holds
\begin{eqnarray}\label{eq:identity}
\int \left\langle \big(\bv\nablax (\bv\otimes \bv)-\bv\nablax (\bu\otimes \bu)\big), \bv-\bu\right\rangle d\bx
 \equiv   \int \left\langle (\bv-\bu), \text{S}[\bu]\,(\bv-\bu)\right\rangle d\bx,  \nonumber
\end{eqnarray}
where $\text{S}[\bu]$ is the symmetric part of the stress tensor $\text{S}[\bu]:= \frac12 (\nablax \bu +\nablax \bu^\top)$.
We therefore have,
\[
\left|\int \left\langle \Lr\nablax (\buN\otimes \buN)-\Lr\nablax (\bu\otimes \bu),(\buN-\bu)\right\rangle d\bx\right|\le |\!|\nablax \bu|\!|_{L^\infty}|\!|\buN-\bu|\!|_{L^2}^2,
\]
The error equation
\[
\ddt (\buN-\bu) + \Lr\nablax(\buN\otimes \buN) - \Lr\nablax (\bu\otimes \bu) = (I-\SN) \Lr\nablax(\buN\otimes \buN),
\]
implies
\begin{eqnarray}
\lefteqn{\frac12\frac{d}{dt}|\!|\buN-\bu|\!|_{L^2}^2} \nonumber  \\
& &\le   |\!|\nablax \bu|\!|_{L^\infty}|\!|\buN-\bu|\!|_{L^2}^2 +\left|\int \left\langle (I-\SN)
[\Lr\nablax (\buN\otimes \buN)],\buN-\bu\right\rangle d\bx\right| \label{eq:Eulerl2}  \\
& & \le   |\!|\nablax \bu|\!|_{L^\infty}|\!|\buN-\bu|\!|_{L^2}^2  +\left|\int \left\langle \left((I-\SN)\nabla\bu\right)\buN,\buN\right\rangle d\bx\right|. \nonumber 
\end{eqnarray}
Arguing along the lines of our convergence statement for Burgers equations we conclude that the following result holds.

\begin{theorem}[{\bf Spectral convergence for smooth solutions of Euler equations}]\label{thm:Euler_spectral} Assume that  for $0<t<T_c$, the solution of the Euler equations \eqref{eq:euler} is smooth, $\bu(\cdot,t)\in L^\infty\left([0,T_c), C^{1+\alpha}(\pil,\pir]\right)$. Then its spectral  Fourier approximation \eqref{eq:Eulerapp} converges in $L^\infty\left([0,T_c], L^2({\mathbb T}^d)\right)$,
\[
\|\buN(\cdot,t)-\bu(\cdot,t)\|_{L^2} \rightarrow 0, \qquad  0\leq t < T_c,
\]
and the following spectral convergence rate estimate holds
\begin{eqnarray*}
\lefteqn{\|\buN(\cdot,t)-\bu(\cdot,t)\|^2} \\
& &  \lesssim e^{\displaystyle {2\!\!\int_0^t|\nabla \bu(\cdot,\tau)|_\infty d\tau}}\!\!\!\left(N^{-2s}\|\bu(\cdot,0)\|^2_{H^s} +
N^{\frac{d}{2}+1-s}\max_{\tau\leq t}\|\bu(\cdot,\tau)\|_{H^s}\right)\!, \ \ s>\frac{d}{2}+1.
\end{eqnarray*}
\end{theorem}
\begin{proof}
Integrating  (\ref{eq:Eulerapp}) against $\buN$ we find the usual statement of $L^2$ energy conservation, 
\[
\|\buN(\cdot,t)\|^2_{L^2}= \|\buN(\cdot,0)\|^2_{L^2}.
\]
Using (\ref{eq:Eulerl2}), we conclude 
\begin{eqnarray*}
\lefteqn{\|\buN(\cdot,t)-\bu(\cdot,t)\|^2_{L^2}  \lesssim  e^{2U'_\infty(t;0)} \|(I-\SN)\bu(\cdot,0)\|^2_{L^2}} \\
& & + \|\buN(\cdot,0)\|_{L^2}^2\int_{0}^t e^{2U'_\infty(t;\tau)}\|(I-\SN)\nabla\bu(\cdot,\tau)\|_{L^\infty} d\tau, \quad U'_\infty(t;\tau):= \int_{s=\tau}^t\|\nablax \bu(\cdot,s)\|_{L^\infty}ds,
\end{eqnarray*}
which yields the spectral convergence rate estimate
\begin{eqnarray}\label{eq:spEulerrate}
\lefteqn{\|\buN(\cdot,t)-\bu(\cdot,t)\|^2_{L^2}}\\
& &  \lesssim e^{2U'_\infty(t;0)} \left(N^{-2s}\|\bu(\cdot,0)\|^2_{H^s} + N^{-s+\frac{d}{2}+1}\max_{\tau\leq t} \|\bu(\cdot,\tau)\|_{H^s}\right), \qquad s>\frac{d}{2}+1. \nonumber
\end{eqnarray}
Observe that the error estimate in the case of Euler equation depends on the truncation error of $\nablax \bu$, corresponding to the dependence on the truncation error of $u_x$ in Burgers equation. The additional loss factor of $d/2$ is due to the $L^\infty({\mathbb T}^d)$-bound, $\max_\bx |(I-\SN){\mathbf w}(\bx)| \lesssim \|{\mathbf w}\|_{H^s}$ for $s>d/2$, consult (\ref{eq:maxHs}).  
\end{proof}

\subsection{The $2/3$ pseudo-spectral approximation of Euler equations}
The pseudo-spectral Fourier method for the Euler equations reads
\[
\ddt \buN +\Lr \nablax \IN(\buN\otimes \buN)=0,
\]
Observe that since $\IN$ does not commute with $\Lr \nablax$, there is no $L^2$-energy conservation.
We introduce the smoothing operator $\SR\buN:=\sum_{|k|\leq m} \sigma_k \widehat{\bu}_k(t)$ which acts on wavenumbers $|k|\leq m=\tthirds N$, while leaving the first $1/3$ portion of the spectrum unchanged: $\sigma_k=\sigma(|k|/N)$, where $\sigma(1-\sigma)$ is supported in $(\frac{1}{3},\tthirds)$. The resulting $2/3$ de-aliasing pseudo-spectral method reads
\begin{equation}\label{eq:Eulerpsapp}
\ddt \buN+\Lr \nablax \IN(\SR\buN\otimes \SR\buN)=0.
\end{equation}
It is the $2/3$ Fourier method which is being used in actual computations, e.g., \cite{OHFS10,KH89,Kerr93,Kerr05} and the references therein. 
Next, we act with the smoothing $\SR$: arguing along the lines of the $2/3$ method for the Burgers' equation in corollary \ref{cor:special},  we find that the   $\bum:=\SR\buN$ satisfies the aliasing-free equation
\begin{equation}\label{eq:Eulerpsp}
\ddt \bum+ \SR\Lr \nablax (\bum\otimes \bum)=0.
\end{equation}
Observe that since  $\SR$ commutes with differentiation, $\bum$ retains incompressibility, 
\[
\ddt \bum+\Lr \nablax \SR (\bum\otimes \bum)=0.
\]
As before, we can integrate against $\buN$ to find by incompressibility of $\bum$, 
\[
\frac12 \dt \int \langle \buN(\bx,t), \bum(\bx,t)\rangle d\bx = -\int \left\langle \SR\buN,  \Lr \nablax (\bum\otimes \bum)\right\rangle d\bx =0,
\]
which implies the weighted $L^2_{\SR}$-energy conservation,
\begin{equation}\label{eq:pspsEulercons}
\|\buN(\cdot,t)\|^2_{L^2_{\SR}} = \|\buN(\cdot,0)\|^2_{L^2_{\SR}}, \qquad \|\buN(\cdot,t)\|^2_{L^2_{\SR}}:=(2\pi)^d\sum \sigma_\bk|\widehat{\bu}_\bk(t)|^2.
\end{equation}

\begin{theorem}[{\bf Spectral convergence of $2/3$ method for smooth Euler solutions}]\label{thm:Euler_pspectral} Assume that  for $0<t<T_c$, the solution of the Euler equations \eqref{eq:euler} is smooth, $\bu(\cdot,t)\in L^\infty\left([0,T_c), C^{1+\alpha}(\pil,\pir]\right)$. Then, the smoothed solution $\bum=\SR\buN$ of its $2/3$ de-aliasing pseudo-spectral Fourier approximation \eqref{eq:Eulerpsapp} converges in $L^\infty\left([0,T_c],L^2({\mathbb T}^d)\right)$,
\[
\|\bum(\cdot,t)-\bu(\cdot,t)\|_{L^2} \rightarrow 0, \qquad  0\leq t < T_c,
\]
and the following spectral convergence rate estimate holds
\begin{eqnarray*}
\lefteqn{\|\bum(\cdot,t)-\bu(\cdot,t)\|^2_{L^2}}\\
& &  \lesssim e^{\displaystyle {2\!\!\int_0^t\!|\nabla \bu(\cdot,\tau)|_\infty d\tau}}\left(N^{-2s}\|\bu(\cdot,0)\|^2_{H^s} +
N^{\frac{d}{2}+1-s}\max_{\tau\leq t}\|\bu(\cdot,\tau)\|_{H^s}\right), \quad s>\frac{d}{2}+1.
\end{eqnarray*}
\end{theorem}
\begin{proof}
We rewrite (\ref{eq:Eulerpsp}) in the form
\[
\ddt \bum+ \Lr \nablax (\bum\otimes \bum)=(I-\SR)\left(\Lr \nablax (\bum\otimes \bum)\right).
\]
Subtract the exact equation (\ref{eq:euler}):  using the identity (\ref{eq:identity}) we find, as before

\begin{eqnarray}\label{eq:Eulerpspl2}
\lefteqn{\quad \frac12\frac{d}{dt}\|\bum-\bu\|_{L^2}^2} \nonumber \\
& & \leq   \|\nablax \bu\|_{L^\infty}\|\bum-\bu\|_{L^2}^2 +\left|\int \left\langle (I-\SR)
[\Lr\nablax (\bum\otimes \bum)],\bum-\bu\right\rangle d\bx\right|\\
& & \leq   \|\nablax \bu\|_{L^\infty}\|\bum-\bu\|_{L^2}^2  +\left|\int \left\langle \left((I-\SR)\nablax\bu\right)\bum,\bum\right\rangle d\bx\right| \nonumber \\
& & \ \ \ + \left|\int \left\langle \left((I-\SR)\nablax\bum\right)\bum,\bum\right\rangle d\bx\right| . \nonumber 
\end{eqnarray}
The last term on the right is due to the fact that $(I-\SR)$ need not annihilate $\nablax\bum$. However, since $\bum$ is incompressible, we find
\begin{eqnarray*}
\int \left\langle \left((I-\SR)\nablax\bum\right)\bum,\bum\right\rangle d\bx & = & \sum_{\alpha,\beta}\int  \bum_\alpha \partial_\alpha\bum_\beta (I-\SR)\bum_\beta d\bx \\
& = & \sum_{\alpha,\beta}\int  \bum_\alpha \hf\partial_\alpha \left(\bum_\beta (I-\SR)\bum_\beta\right)d\bx \\
& = & -\hf  \int  \sum_\alpha \partial_\alpha \bum_\alpha \sum_\beta \left(\bum_\beta (I-\SR)\bum_\beta\right)d\bx =0.
\end{eqnarray*}
We end up with the error bound
\begin{eqnarray*}
\lefteqn{\|\bum(\cdot,t)-\bu(\cdot,t)\|^2_{L^2} \lesssim  e^{2{U'_\infty(t;0)}} \|(I-\SR)\bu(\cdot,0)\|^2_{L^2}} \\
& & + \|\bum(\cdot,0)\|_{L_{\SR}^2}^2\int_{0}^t \!\!e^{{2U'_\infty(t;\tau)}}\|(I-\SR)\nablax\bu(\cdot,\tau)\|_{L^\infty} d\tau, 
\quad U'_\infty(t;\tau):=\int_{s=\tau}^t\!\!\!\!\|\nablax \bu(\cdot,s)\|_{L^\infty}ds,
\end{eqnarray*}
and  spectral convergence rate follows.
\end{proof}

\ifx
\begin{remark}
\myb{We note that the $L^2$-error estimates stated in theorems \ref{thm:Euler_spectral} and \ref{thm:Euler_pspectral} do not imply the convergence in more regular norm and therefore do not exclude  some oscillations.
On the other hand the (pseudo-) spectral approximations  conserves the energy and therefore anomalous dissipation of energy should not appears even if the exact solution $u(t)$ does not remain smooth.} \}\}\}
\end{remark}
\fi

\section{Fourier method for Euler equations: failure of convergence for weak solutions?}\label{sec:nonsmoothEuler}

We now consider the convergence of the $2/3$ method (\ref{eq:Eulerpsapp}) for weak solutions of Euler equations. Its $m$-mode de-aliased  solution is governed by (\ref{eq:Eulerpsp})
\begin{equation}\label{eq:Eulerpspb}
\ddt \bum + \SR\Lr\nablax \left(\bum\otimes \bum\right)=0.
\end{equation}
The method is energy preserving in the sense that $\SR^{1/2}\buN$ is $L^2$-conservative, (\ref{eq:pspsEulercons}), and hence $\bum=\SR\buN$ has s a weak limit, $\bubar$.  The question is to characterize whether $\bubar(x,t)$ is an energy conserving weak solution of Euler equations (\ref{eq:euler}), 
\begin{equation}\label{eq:cdeu}
\ddt \bubar +\Lr\nablax (\bubar\otimes \bubar)=0.
\end{equation}
To this end we compare (\ref{eq:Eulerpsapp}) and (\ref{eq:cdeu}): since $\bum$ tends weakly to $\bubar$ and $\partial_t \bum \rightharpoonup \partial_t\bubar$, then comparing the remaining spatial parts of (\ref{eq:Eulerpsapp}) and (\ref{eq:cdeu}), yields that $\SR\Lr[\bum\otimes \bum](x,t)$ and hence
$\Lr[\bum\otimes \bum](x,t)$  tends weakly to $\Lr[\bubar\otimes\bubar](x,t)$. 
This, however, is not enough to imply the strong convergence of $\uN$, as shown by a simple 
  counterexample of  a 2D potential flow, $\bu_n=\nablax^\perp \Phi_n$ where 
\[
\Phi_n(x_1,x_2)  = \frac1n\Xi(x_1,x_2)(\sin n x_1+ \sin n x_2) 
\]
with $\Xi(x_1,x_2)\in \mathcal{D} (\R^2)$  localized near any point (say $(0,0)$) with weak limit $\bubar \equiv 0$.
In this case $w\,\mbox{-}\lim_{N\rightarrow \infty} \nabla \Lr(\uN\otimes \uN)= \nabla \Lr(\bubar\otimes \bubar) = 0$, yet
\[
w\,\mbox{-}\!\!\!\lim_{N\rightarrow \infty}  ({\uN}_1)^2= w\,\mbox{-}\!\!\!\lim_{N\rightarrow \infty}  ({\uN}_2)^2 = \frac{\Xi(x_1,x_2)^2}2 \neq 0.
\]
Although $\bubar$ need not be a weak solution of Euler equations, it satisfies a weaker notion of a \emph{dissipative solution} in the sense of DiPerna-Lions \cite{Lions96}
To this end, let $\bw$ a   divergence-free smooth solution of
\begin{equation}\label{eq:Eulerw}
\del_t \bw +  \Lr(\nabla \bw\otimes \bw ) =E(\bw), \qquad \Lr E(\bw)=0.
\end{equation}
Now, compare it with the $2/3$ solution  (\ref{eq:Eulerpspb}): the same computation with Gronwall lemma leads to,
\begin{eqnarray*}
\lefteqn{\|(\buN-\bw)(\cdot,t)\|^2_{L^2(\Omega)}  \le  e^{2 {W'_\infty(t;0)}} |\!|(\buN-\bw)(\cdot,0) |\!|^2_{L^2(\Omega)}} \\
&  & \quad +  2\| \buN(\cdot,0)\|^2_{L^2(\Omega)} \int_0^t
  \|(P_N\bw-\bw))(\cdot,s) \|_{W^{1,\infty}(\Omega)} \\
&  & \quad + 2 \int_0^t \!\!\!e^{2 {W'_\infty(t;\tau)}} \|(E(\bw(\tau)), \buN(\tau)-\bw(\tau))\|  d\tau, \quad W'_\infty(t;\tau):=\int_{s=\tau}^t\!\!\!\!\|\nablax \bw(\cdot,s)\|_{L^\infty(\Omega)}ds.
\end{eqnarray*} 
Passing to the weak limit it follows that $\bubar$ is a dissipative solution, satisfying for all  divergence-free smooth solution of (\ref{eq:Eulerw}), the stability estimate 
\begin{eqnarray*}
 \|(\bubar-\bw)(\cdot,t)\|^2_{L^2(\Omega)} & \le & e^{2{W'_\infty(t;0)}} \|(\bubar  -\bw)(\cdot,0) \|^2_{L^2(\Omega)} \\
 & + & 2\int_0^t e^{2 {W'_\infty(t;\tau)}}\,|(E(\bw(\tau)), \bubar(\tau) -\bw(\tau)))|d\tau.
\end{eqnarray*}

The notion  of dissipative solution can be instrumental in the context of  stability near a smooth solution, $\bw$,  or even in the context of uniqueness. However, the construction of  \cite{DLSz10} does not exclude the existence of rough initial data for which the Cauchy problem associated with Euler equations (\ref{eq:euler}) have an infinite set of dissipative solutions. In fact, it is observed in \cite{DLSz10}  that any weak solution with a non-increasing energy, $\|\bu(\cdot,t)\|_{L^2} \leq \|\bu(\cdot,0)\|_{L^2}$,  
is a dissipative solution. These, so-called  admissible solutions, arise as solutions of the Cauchy problem for an infinite set of (rough) initial data, and can be obtained as strong limit in $C(0,T ;L^2_{\rm{weak}}(\Omega))$ of solutions for the problem 
\[
\del_t \buN + \Lr(\nablax (\buN\otimes \buN)= E_N
\] 
with $w\,\mbox{-}lim E_N = 0$, while $\displaystyle \int \langle E_N,\buN \rangle dx $ does not converge to $0$. 

We  summarize the above observations, by stating that  as long as the solution of the Euler equations remains sufficiently smooth, then its spectral and de-aliased pseudo-spectral approximations converge  in $L^2(\Omega)$. 
Indeed, in theorems \ref{thm:Euler_spectral} and \ref{thm:Euler_pspectral}, we quantified the convergence rate for $H^s$-regular solutions $\bu$. If $\bu$ has a minimal smoothness such that the vorticity \mbox{\boldmath$\omega$}${}_N:=\nabla \times \buN$ is compactly embedded in $C([0,T],H^{-1}(\R^N))$, then by the div-curl lemma, $\buN(\cdot,t)$ converges strongly in $L^\infty([0,T],L^2_{loc}(\R^N))$ to an energy-preserving limit solution $\bubar$, \cite{LNT00}.

\smallskip\noindent
The situation is different, however, when dealing with ``rough'' solutions of the underlying Euler equations.
In the absence of any information re:the smoothness of the underlying Euler solutions (--- as loss of smoothness for the $3D$ Euler equations is still a challenging open problem),  energy-preserving numerical method need not shed light on  the question of global regularity vs. finite-time blow-up.
Recall that $L^2$-energy conservation was conjectured by Onsager \cite{On49} and verified  in \cite{Ey94,CET94,BT10} under the assumption of  minimal smoothness of $\bu$, but  otherwise is not supported by the energy decreasing  solutions of Euler equation, \cite{Co07,DeLS12, Buck13}.

\smallskip\noindent
The similar scenario of quadratic entropy conservation in the context of  Burgers' equations, is responsible for spurious oscillations, and  its detailed analysis can be found in \cite{La86} after \cite{vN63}. Here, enforcing energy conservation at the ``critical'' time when Euler solutions seem to lose sufficient smoothness leads to nonlinear  instability which manifests itself through oscillations noticeable throughout the computational domain, in agreement with the numerical evidence observed in \cite{HL07}, see Figure \ref{fig:d}(a) below. The precise large-time behavior of the (pseudo-) spectral approximations is intimately related to a proper albeit yet unclear notion of propagating smoothness for solutions of Euler equations which, even if they do not explicitly blow up, may exhibit spurious oscillations due to the amplification factor in higher norms.

\begin{figure}[ht]
\begin{center}
\begin{tabular}{cc}
\includegraphics[scale=0.6, height=100pt, width=214pt]{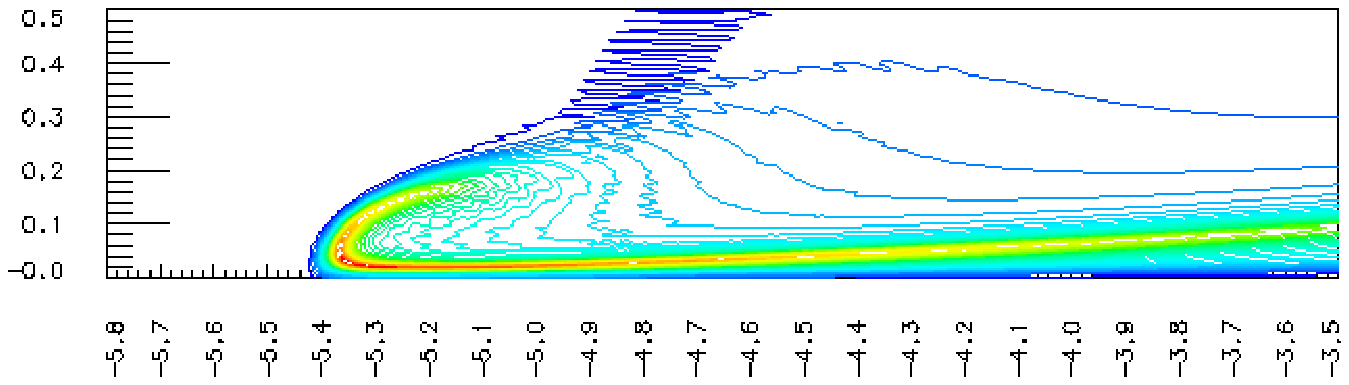} &
\includegraphics[scale=0.6, height=100pt, width=214pt]{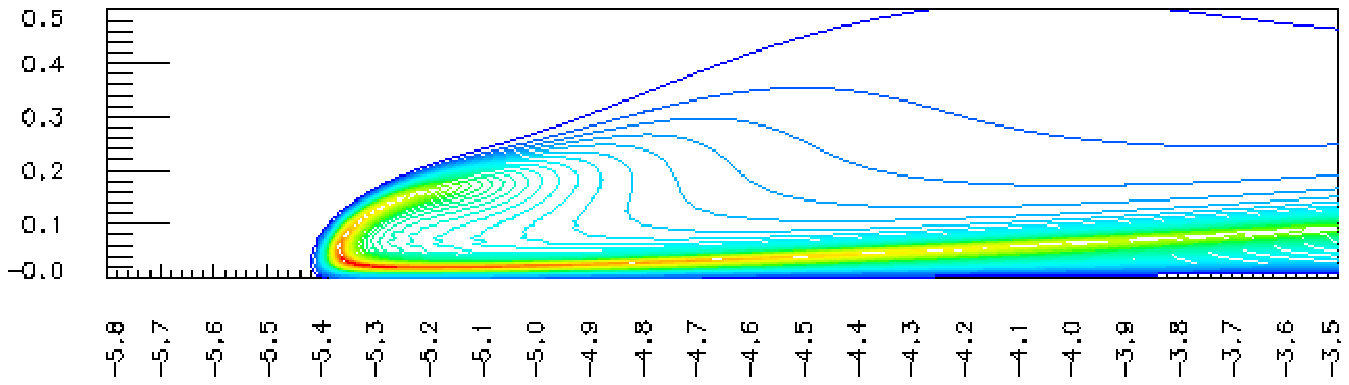} \\
\includegraphics[scale=0.6, height=100pt, width=214pt]{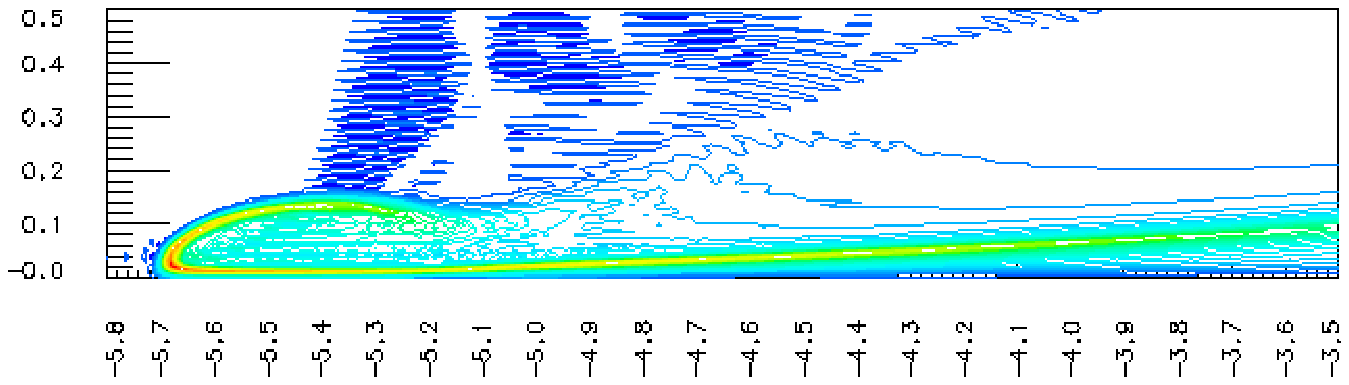} &
\includegraphics[scale=0.6, height=100pt, width=214pt]{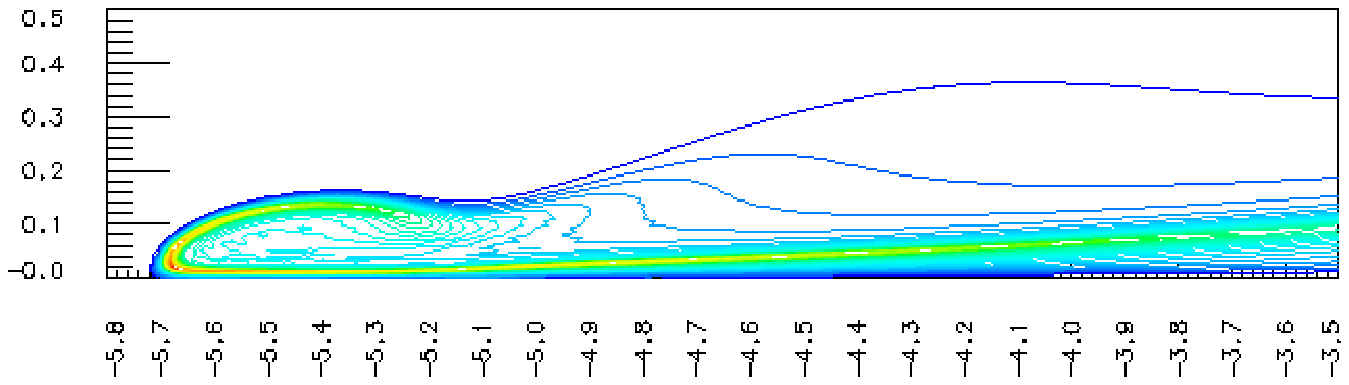}\\
(a) & (b)  
\end{tabular}
\end{center} 
\caption{A comparison of axial vorticity contours of 3D Euler computation \cite{HL07} at $t=18$ (top) and at $t=19$ (bottom). On left(a): the solution obtained by the energy preserving $\frac{2}{3}$ de-aliasing method; on right (b): the solution obtained by an energy decreasing smoothing of the Fourier method. The resolution is $1024\times 768\times 2048$.}\label{fig:d}
\end{figure}

\section{The spectral viscosity method: nonlinear stability and spectral convergence}\label{sec:sv}
The  nonlinear instability results   in sections \ref{sec:2-3rd} and \ref{sec:nonsmoothEuler} emphasize  the  competition between spectral convergence for 
smooth solutions vs. nonlinear instabilities for problems which lack sufficient smoothness. 
One class of methods  for nonlinear evolution equations which entertain both --- spectral convergence and nonlinear stability,  is the class \emph{spectral viscosity} (SV) methods, introduced in \cite{Tad89}.  We demonstrate the SV method in the context of Burgers equation, 

\begin{subequations}\label{eqs:SVb}
\begin{equation}\label{eq:SVb}
\frac{\partial}{\partial t}\uN(x,t)+\frac{1}{2}\frac{\partial}{\partial x}\Big( \IN\left[\uN^2\right](x,t) \Big)=SV[\uN](x,t), \qquad x\in {\mathbb T}([0,2\pi)).
\end{equation}
On the right of (\ref{eq:SVb}) we have added a judicious amount of  spectral viscosity of order $2r$:
\begin{equation}\label{eq:SV}
SV[\uN](x,t):= -N\sum_{|k|\leq N}\sigma\left(\frac{|k|}{N}\right)\widehat{u}_k(t)e^{ikx}, \quad \sigma(\xi)\lesssim 
\left(|\xi|^{2r}-\frac{1}{N}\right)_+, \qquad r\geq1
\end{equation}
\end{subequations}
Without it, the pseudo-spectral solution will develops spurious Gibbs oscillations after the formation of shocks. 
Observe that the spectral viscosity term in (\ref{eq:SV}) adds a \emph{spectrally small} amount of numerical
dissipation for high modes, $k \gg 1$ (in contrast for ''standard'' finite-order amount of numerical dissipation 
in finite-difference methods),
\[
\|SV[w]\|_{\dot{H}^\alpha} \lesssim N^{1-(\alpha-\beta)(1-\frac{1}{2r})}\|w\|_{\dot{H}^\beta}, \qquad \forall \beta \ll \alpha-1 \in {\mathbb R}.
\]
 Indeed, the  low-pass SV filter on the right of (\ref{eq:SVb})  vanishes for modes $|k|\leq {N}^{(2r-1)/2r}$,  which in turn leads to spectral convergence for smooth solutions. Arguing along the lines of theorem \ref{thm:pspsmooth} we state the following.

\begin{theorem}[{\bf Spectral convergence for smooth solutions of Burgers' equations}]\label{thm:svsmoothb} Consider the Burgers equation, \eqref{eq:burgers}, with a smooth solution  $u(\cdot,t)\in L^\infty\left([0,T_c], C^{1+\alpha}(\pil,\pir]\right)$. Then its spectral viscosity approximation \eqref{eqs:SVb},
\[
\frac{d}{dt}\uN(x_\nu,t)+\frac{1}{2}\frac{\partial}{\partial x}\left( \IN\left[\uN^2\right](x,t) \right)_{\big|x=x_\nu}=SV[\uN](x_\nu,t), \qquad \nu=0, 1, \ldots, 2N.
\]
 converges, $\|\uN(\cdot,t)-u(\cdot,t)\|_{L^2} \rightarrow 0$ for $0\leq t \leq T_c$ and the following spectral convergence rate estimate holds for all $s>\threehf$,
\begin{eqnarray*}
\lefteqn{\|\uN(\cdot,t)-u(\cdot,t)\|^2}\\
& &  \lesssim e^{\displaystyle {\int_0^t|u_x(\cdot,\tau)|_\infty d\tau}} \left(N^{-2s}\|u(\cdot,0)\|^2_{H^s} + N^{\frac{2r-1}{2r}(\threehf-s)}\max_{\tau\leq t}\|u(\cdot,\tau)\|_{H^s}\right)\!, \quad s>\threehf.
\end{eqnarray*} 
\end{theorem}

At the same time, spectral viscosity is strong enough to enforce a sufficient amount of $L^2$ energy dissipation, which in turn implies convergence after the formation of shock discontinuities.  We quote below the convergence statement of  the hyper-SV method.

\begin{theorem}[{\bf Convergence of the hyper-SV method for Burgers equation} \cite{Tad89,Tad93b,Tad04}]\label{thm:svshockb} Let $u$ be the unique entropy solution of the inviscid Burgers equation, \eqref{eq:burgers}, subject to uniformly bounded initial data $u_0$, and let $\uN$ be  the spectral viscosity approximation  \eqref{eqs:SVb} subject to $L^\infty$ data $\uN(0)\approx u_0$. Then, if $\uN$ remains uniformly bounded\footnote{The question of uniform boundedness of $\uN$ was proved for the second order SV method, corresponding to $r=1$,  in \cite{Tad93a}, but it remains open for the hyper SV case with $r>1$.} it converges to the unique entropy solution, $\|\uN(\cdot,t)-u(\cdot,t)\|_{L^2} \rightarrow 0$.
\end{theorem}

\begin{remark} We note that unlike the $2/3$ de-aliasing method, the SV method does \emph{not} completely remove the high-frequencies but instead, it  introduces ``just the right amount'' of smoothing for $|k|\gg1$ which enables to balance spectral accuracy with nonlinear stability. The SV method can be viewed as a proper smoothing which addresses the instability of general smoothing of the
pseudo-spectral Fourier method sought in remark \ref{rem:smoo}.
Moreover, even after the formation of shock discontinuities, the SV solution still contains 
highly accurate information of the exact entropy solution which can be extracted by post-processing, \cite{SW95}.   
\end{remark}

Similar results of spectral convergence of  SV methods hold in the context of  incompressible Euler equations, \cite{KK00,SS07,AX09},
\begin{eqnarray}\label{eq:sveu}
\lefteqn{\ddt \buN + \Lr\nablax \IN\left(\SR\buN\otimes \SR\buN\right)=SV[\buN],} \\
& & \qquad \qquad SV[\buN](\bx,t):= -N\sum_{|\bk|\leq N}\sigma\left(\frac{|\bk|}{N}\right)\widehat{\bu}_\bk(t)e^{i\bk\cdot\bx}. \nonumber
\end{eqnarray}
In contrast to the spurious oscillations with the $2/3$ methods shown in figure \ref{fig:d}(a), the oscillations-free results in \ref{fig:d}(b) correspond to the proper amount of smoothing employed in \cite{HL07}. Thus, the issue  of adding ``just the right amount'' of hyper-viscosity is particularly relevant in this context of Large Eddy Simulation (LES) for highly turbulent flows, when one needs to strike a balance between a sufficient amount of numerical dissipation for stability without giving up on high-order accuracy for physically relevant Euler (and Navier-Stokes solutions). The SV method in (\ref{eq:sveu}) adds this balanced amount of hyper-viscoisty,  \cite{KK00,GP03,SK04,SS07,PSSBS07}.

\section{Beyond quadratic nonlinearities: 1D isentropic equations}\label{sec:isen}
We consider the one-dimensional isentropic equations in Lagrangian coordinate,
\begin{subequations}\label{eqs:isen}
\begin{eqnarray}
&& \ddt u +\ddx  q(v)=0, \quad q'(v)>0\\
&&\ddt v+\ddx  u=0,
\end{eqnarray}
\end{subequations}

which is approximated by the spectral method
\begin{subequations}\label{eqs:specisen}
\begin{eqnarray}
&&\ddt u_N +\ddx q(v_N)=(I-\SN)q(v_N),\\
&&\ddt v_N+\ddx u_N=0.
\end{eqnarray}
\end{subequations}

Denote by $U$ the vector of conservative variables, $U:=(u,v)^\top$, by $F(U)$ the corresponding flux,  $F(U):=(q(v),u)^\top$ and let $\eta(U)$ be the entropy  $\eta(U):=\frac12|u|^2+ Q(v), \ Q'(v)=q(v)$. 
Multiplying the system by $\nabla_U \eta(U)$ and integrating gives:
\[
\frac d{dt} \int\left(\frac{|\uN|^2}2 + \uN \del_x q(v_N)+ q(v_N)\del_x u_N\right)dx=\int (I-\SN)q(v_N)u_Ndx=0
\]
and hence there the total entropy is conserved for both the exact an approximate solutions\footnote{This intriguing property seems specific to the isentropic equation in Lagrangian coordinate.}
\[
\del_t\int \eta(U)dx=0 \quad \hbox{and} \quad \del_t\int \eta(U_N)dx=0.
\]
Continuing as in DiPerna-Chen  \cite{DiP83,Ch97,Daf79}, we write
\begin{eqnarray}
\lefteqn{\del_t\int \left(\eta(U_N)-\eta(U)-\left\langle \eta'(U),U_n-U\right\rangle \right)dx=} \nonumber \\
&& \quad \int \left\langle \eta''(U)U_t, (U_N-U)\right\rangle dx- \int \left\langle \eta'(U),(U_N)_t-U_t \right\rangle dx= \label{eq:abx}\\
&&-\int \left\langle \eta''(U)F(U)_x, U_N-U\right\rangle dx-\int \left\langle \eta'(U), F(U_N))_x- F(U)_x\right\rangle dx+ \hbox{error term}\nonumber \\
& &  =:{\mathcal I}_1+{\mathcal I}_2+{\mathcal I}_3\nonumber
\end{eqnarray}
The first two terms on the right amount to
\begin{eqnarray*}
|{\mathcal I}_1+{\mathcal I}_2|&= &\Big|\int \left\langle \eta''(U) F(U)_x,  U_N-U \right\rangle dx+\int \left\langle \eta'(U), F(U_N)_x-F(U)_x \right\rangle dx\Big|\\
& =&\Big|\int \left\langle \eta''(U) F'(U) U_x, U_N-U\right\rangle dx -\left\langle  \eta''(U)U_x, F(U_N)-F(U)\right\rangle dx\Big|\\
& = &\Big|\int \left\langle \eta''(U) F'(U) U_x, U_N-U\right\rangle dx -\left\langle  \eta''(U)U_x, F'(U)U_x +{\mathcal O}\|U_N-U\|^2\right\rangle dx\Big|.
\end{eqnarray*}
Since the entropy Hessian symmetrize the system, one has $\eta''(U) F'(U)= F'(U) \eta''(U)$, and we conclude that 
 the last expression does not exceed
$$
|{\mathcal I}_1+{\mathcal I}_2|= \Big|\int \left\langle \eta''(U) F'(U) U_x, U_N-U\right\rangle dx -\left\langle \eta''(U) U_x, F(U_N)-F(U)\right\rangle dx\Big|\lesssim \|U\|_{C^1}\|U_N-U\|^2
$$
On the other hand 
$$
{\mathcal I}_3= \hbox{error term}=\int (I-\SN)q_x(v_N)(u-u_N) dx =\int \del_x q(v_N) (I-\SN)u_xdx 
$$
which goes to zero for sufficiently smooth $u \in C^{1+\alpha}$. Inserting the last two bound into (\ref{eq:abx}) we find that
\[
\del_t\int \left(\eta(U_N)-\eta(U)-\left\langle \eta'(U),U_n-U\right\rangle \right)dx \lesssim \|U\|_{C^1}\|U_N-U\|^2 + {o}(1).
\]
By strict convexity, the integrand on the left is of order $\sim \|U_N-U\|^2$ and we conclude the following.
\begin{theorem} Assume that  for $0<t<T_c$, the solution of the isentropic Euler equations \eqref{eqs:isen} is smooth, $U(\cdot,t)\in L^\infty\left([0,T_c), C^{1+\alpha}(\pil,\pir]\right)$. Then, its spectral  approximation \eqref{eqs:specisen} converge in $L^\infty_t L^2_x$,
\[
\|U_N(\cdot,t)-U(\cdot,t)\|_{L^2} \rightarrow 0, \qquad  0\leq t < T_c.
\]
\end{theorem}

\bibliographystyle{plain}

\end{document}